\documentclass[a4paper,11pt]{article}
\usepackage{graphicx,epsfig,picins}
\usepackage{fancyhdr,fancybox}
\usepackage{indentfirst}
\usepackage{titlesec}
\usepackage{verbatim}
\usepackage[sort&compress, numbers]{natbib}
\usepackage{array,dcolumn,tabularx}
\usepackage{setspace}
\usepackage{geometry}
\usepackage{extarrows,chemarrow,xypic} 
\usepackage[small]{caption2}

\usepackage{times}     

\usepackage{latexsym}
\usepackage{amsmath}                 
\usepackage{amssymb}                 
\usepackage{amsbsy}
\usepackage{amsthm}
\usepackage{amsfonts}
\usepackage{mathrsfs}                
\usepackage{bm}                      
\usepackage{relsize}                 
\usepackage{hyperref}  

\newcommand{\paperfont}{\fontsize{11pt}{1.2\baselineskip}\selectfont}
\begin{document}

\theoremstyle{definition}
\makeatletter
\thm@headfont{\bf}
\makeatother
\newtheorem{definition}{Definition}
\newtheorem{example}{Example}
\newtheorem{theorem}{Theorem}
\newtheorem{lemma}{Lemma}
\newtheorem{corollary}{Corollary}
\newtheorem{remark}{Remark}
\newtheorem{proposition}{Proposition}

\lhead{}
\rhead{}
\lfoot{}
\rfoot{}

\renewcommand{\refname}{References}
\renewcommand{\figurename}{Figure}
\renewcommand{\tablename}{Table}
\renewcommand{\proofname}{Proof}

\newcommand{\diag}{\mathrm{diag}}

\numberwithin{equation}{section}

\title{\textbf{Reduction of Markov chains with two-time-scale state transitions}}
\author{Chen Jia$^1$ \\
\footnotesize $^1$School of Mathematical Sciences, Peking University, Beijing 100871, China.\\
\footnotesize Email address: jiac@pku.edu.cn}
\date{}                              
\maketitle                           
\thispagestyle{empty}                

\paperfont

\begin{abstract}
In this paper, we consider a general class of two-time-scale Markov chains whose transition rate matrix depends on a parameter $\lambda>0$. We assume that some transition rates of the Markov chain will tend to infinity as $\lambda\rightarrow\infty$. We divide the state space of the Markov chain $X$ into a fast state space and a slow state space and define a reduced chain $Y$ on the slow state space. Our main result is that the distribution of the original chain $X$ will converge in total variation distance to that of the reduced chain $Y$ uniformly in time $t$ as $\lambda\rightarrow\infty$. \\


\noindent 
\textbf{Classifications}: 60J27, 60J28, 93C70
\end{abstract}

\section{Introduction}
In many areas of natural sciences, we often encounter systems that can be modeled by the following coupled stochastic differential equation with two separate time scales:
\begin{equation}\label{SDE}\left\{
\begin{split}
dX_t &= b_1(X_t,Y_t)dt + \sigma_1(X_t,Y_t)dW_t, \\
dY_t &= \lambda^2b_2(X_t,Y_t)dt + \lambda\sigma_2(X_t,Y_t)dW_t,
\end{split}\right.
\end{equation}
where $X_t\in R^k$, $Y_t\in R^{n-k}$, and $\lambda>0$ is a parameter. When $\lambda$ is very large, the components of $X_t$ are slow variables and the components of $Y_t$ are fast variables. Roughly speaking, if we focus on the dynamics of the slow variables $X_t$, then the fast variables $Y_t$ can be averaged out. In this way, we can reduce the original $n$-dimensional stochastic differential equation to a simpler $k$-dimensional stochastic differential equation. This topic has been discussed thoroughly \cite{kifer1992averaging, veretennikov1999large, imkeller2001stochastic, khasminskii2005limit}.

In the above problem, the phase space of the stochastic system is the Euclidean space. However, in a number of problems arising in physics, chemistry, biology, and engineering \cite{cornish1995fundamentals, keener1998mathematical, beard2008chemical}, we frequently encounter stochastic systems that can be modeled by continuous-time Markov chains with a discrete state space whose state transitions have two separate time scales. Specifically, let $X=\{X_t: t\geq 0\}$ be a continuous-time Markov chain with finite state space $S$ in which some transition rates are much larger than the other ones. Within this framework, the transition rates between states have two separate time scales.

In order to study this type of two-time-scale Markov chains, engineers proposed the concept of stiff Markov chains \cite{bobbio1986aggregation, reibman1989analysis, bobbio1990computing, malhotra1994stiffness}. Roughly speaking, let $\alpha$ be a given threshold value. If a transition rate of the Markov chain $X$ is larger or smaller than $\alpha$, then this rate is called a fast or slow rate, respectively. The state space $S$ of the Markov chain $X$ can be further divided into a fast state space $A$ and a slow state space $B$. Since the holding times of the fast states are much shorter than those of the slow states, we have good reasons to believed that the original chain $X$ on the state space $S$ can be reduced to a simpler Markov chain $Y$ on the slow state space $B$. Although engineers have studied the approximation algorithm for stiff Markov chains, they did not obtain any rigorous mathematical results about the asymptotic behavior of stiff Markov chains since the choice of the threshold value $\alpha$ is rather arbitrary.

In addition, Yin, Zhang, and coworkers \cite{khashinskii1996asymptotic, yin2000asymptotic, yin2000singularly, yin2007singularly} have done a systematic study on an important class of two-time-scale Markov chains named as singularly perturbed Markov chains, and these materials have been organized into a textbook recently \cite{yin2012continuous}. In a singularly perturbed Markov chain $X$, the transition rate matrix $Q(\lambda)$ depends on a parameter $\lambda>0$ in a linear way:
\begin{equation}
Q(\lambda) = \lambda\widetilde{Q}+\widehat{Q},
\end{equation}
where $\widetilde{Q}$ and $\widehat{Q}$ are two transition rate matrices. When $\lambda$ is very large, the transition rate matrix $\widetilde{Q}$ governs the rapidly changing components and the transition rate matrix $\widehat{Q}$ governs the slowly changing ones. In the singularly perturbed literature, the authors used the analytic approach of matched asymptotic expansions from singular perturbation theory to construct approximate sequences for the distribution of the Markov chain $X$. The authors proved that the distribution of the Markov chain $X$ at time $t$ will converge to the so-called zero-order outer expansion $\phi(t)$ as $\lambda\rightarrow\infty$ for any $t>0$.

In some areas of natural sciences such as biochemistry and biophysics, we frequently encounter chemical reaction systems that can be modeled by continuous-time Markov chains whose transition rate matrix depends on a parameter $\lambda>0$, which usually represents the concentration of a molecule \cite{cornish1995fundamentals, keener1998mathematical, beard2008chemical}. In these systems, however, the transition rate matrix $Q(\lambda)$ in general does not depend on $\lambda$ in a linear way. Therefore, we need to study the asymptotic behavior of general two-time-scale Markov chains that cannot be described by singularly perturbed Markov chains. Although this problem is fundamental and important, there is still a lack of rigorous mathematical results about this problem. The aim of this paper is to fill in this gap.

In this paper, we consider a general class of two-time-scale Markov chains whose transition rate matrix $Q(\lambda)$ depends on a parameter $\lambda>0$. We assume that some transition rates of the Markov chain $X$ will tend to infinity as $\lambda\rightarrow\infty$. Similar to the consideration in stiff Markov chains, we divide the state space $S$ into a fast state space $A$ and a slow state space $B$. Moreover, we define a reduced chain $Y=\{Y_t: t\geq 0\}$ on the slow state space $B$. We then use a purely probabilistic approach to prove that the distribution of the original chain $X$ will converge in total variation distance to that of the reduced chain $Y$ uniformly in time $t$ as $\lambda\rightarrow\infty$ (see Corollaries \ref{dtvlocal} and \ref{dtvglobal}). Specifically, if the initial distribution $\pi$ of the original chain $X$ is concentrated on the low state space $B$, then we prove that for any $T>0$,
\begin{equation}\label{convergence}
\lim_{\lambda\rightarrow\infty}\sup_{0\leq t\leq T}d_{TV}(P^\lambda_{\pi}(X_t\in\cdot),P_{\pi}(Y_t\in\cdot)) = 0,
\end{equation}
where $P^\lambda_{\pi}$ denotes the probability measure under transition rate matrix $Q(\lambda)$ and initial distribution $\pi$, and $d_{TV}$ represents the total variation distance. This result shows that if the initial distribution of the original chain $X$ is concentrated on the slow state space $B$, then the distributions of the original chain $X$ and the reduced chain $Y$ will be close to each other over any finite time interval when $\lambda$ is sufficiently large. The readers may ask whether the above approximation theorem holds not only over any finite time interval, but also over the whole time axis. In general, the answer is false (see Example \ref{counter}). However, we prove a satisfying result that if the reduced chain $Y$ is irreducible, then the convergence is uniform over the whole time axis, that is,
\begin{equation}
\lim_{\lambda\rightarrow\infty}\sup_{t\geq 0}d_{TV}(P^\lambda_{\pi}(X_t\in\cdot),P_{\pi}(Y_t\in\cdot)) = 0.
\end{equation}

Moreover, we also study the asymptotic behavior of the Markov chain $X$ under general initial distributions and obtain the corresponding approximation theorems (see Theorems \ref{weaklocal} and \ref{weakglobal}). If the initial distribution of the original chain $X$ is not concentrated on the slow state space $B$, we cannot expect that the distributions of the original chain $X$ and the reduced chain $Y$ are close to each other over the whole time axis. However, we prove that although the initial distribution may not be concentrated on $B$, the distribution of the original chain $X$ will be ``almost" concentrated on $B$ after a very short time when $\lambda$ is sufficiently large (see Theorem \ref{shortterm1}). Based on this fact, we prove that for any $h>0$, the distributions of the original chain $X$ and the reduced chain $Y$ are close to each other after time $h$ when $\lambda$ is sufficiently large. Specifically, if the initial distribution $\pi$ of the original chain $X$ is not concentrated on the slow state space $B$, then we prove that for any $0<h<T$,
\begin{equation}
\lim_{\lambda\rightarrow\infty}\sup_{h\leq t\leq T}d_{TV}(P^\lambda_{\pi}(X_t\in\cdot),P_{\gamma(\pi)}(Y_t\in\cdot)) = 0,
\end{equation}
where $\gamma(\pi)$ is a probability distribution concentrated on the slow state space $B$. If the reduced chain $Y$ is further assumed to be irreducible, then the above convergence can be strengthened as follows:
\begin{equation}
\lim_{\lambda\rightarrow\infty}\sup_{t\geq h}d_{TV}(P^\lambda_{\pi}(X_t\in\cdot),P_{\gamma(\pi)}(Y_t\in\cdot)) = 0.
\end{equation}

At the end of this paper, we study the relationship between our work and the theory of singularly perturbed Markov chains in great detail. We hope that the approximation theorems established in this paper can give enlightenment to both theoretical analysis and numerical simulation of stochastic systems modeled by two-time-scale Markov chains arising in physics, chemistry, biology, and engineering.

\section{Preliminaries}
In this paper, we consider a continuous-time Markov chain $ X=\{X_t:t\geq 0\}$ on the finite state space $S=\{1,2,\cdots,|S|\}$ with transition rate matrix $Q(\lambda)=(q_{ij}(\lambda))$ which depends on a parameter $\lambda>0$. The finiteness of the state space $S$ is essential to establishing the main results of this paper. For simplicity, we assume that the transition rate matrix $Q(\lambda)$ is irreducible for each $\lambda>0$. We further assume that $\lim_{\lambda\rightarrow\infty}q_{ij}(\lambda)$ is finite or $\lim_{\lambda\rightarrow\infty}q_{ij}(\lambda) = \infty$ for any pair of states $i\neq j$. When $\lambda$ is sufficiently large, this framework just describes a Markov chain whose state transitions have two separate time scales. We shall study the limit behavior of the Markov chain $X$ as $\lambda\rightarrow\infty$.

Consistent with standard notations \cite{norris1998markov}, we set $q_i(\lambda) = -q_{ii}(\lambda) = \sum_{j\neq i}q_{ij}(\lambda)$ and set $q_i = \lim_{\lambda\rightarrow\infty}q_i(\lambda)$. According to whether these $q_i$ are finite or infinity, we can classify the state space $S$ into two disjoint subsets.
\begin{definition}\label{states}
Let $q_i = \lim_{\lambda\rightarrow\infty}q_i(\lambda)$. \\
(1) If $q_i = \infty$, then $i$ is called a fast state. The set of all fast states is denoted by $A$. \\
(2) If $q_i < \infty$, then $i$ is called a slow state. The set of all slow states is denoted by $B$.
\end{definition}

Obviously, we have $A\cap B = \emptyset$ and $A\cup B = S$. If $i$ is a fast state, then $q_i(\lambda)$ will be very large when $\lambda$ is sufficiently large. Recall that $q_i(\lambda)$ is the rate of the exponential holding time of state $i$. This means that the holding times of state $i$ will be very short when $\lambda$ is sufficiently large. That is why we name such a state $i$ as a fast state. According to the above definition, $\lim_{\lambda\rightarrow\infty}q_{ij}(\lambda)$ will be finite for any state $j$ if $i$ is a slow state. In the following discussion, we always assume that the slow state space $B\neq\emptyset$ and assume that $q_i>0$ for each $i\in B$.

\begin{example}
Figure \ref{MWCmodel} illustrates a Markov chain $X$ for which two transitions rates depend on $\lambda$ in a linear way and other transition rates are independent of $\lambda$. This Markov chain, which is referred to as the Monod-Wyman-Changeux allosteric model \cite{monod1965nature, changeux2012allostery}, is important in biochemistry and biophysics since it is widely used to model the conformational changes of receptors in living cells. According to Definition \ref{states}, the fast state space is $A = \{1,2\}$ and the slow state space is $B = \{3,4\}$.
\end{example}
\begin{figure}[!htb]
\begin{center}
\centerline{\includegraphics[width=0.3\textwidth]{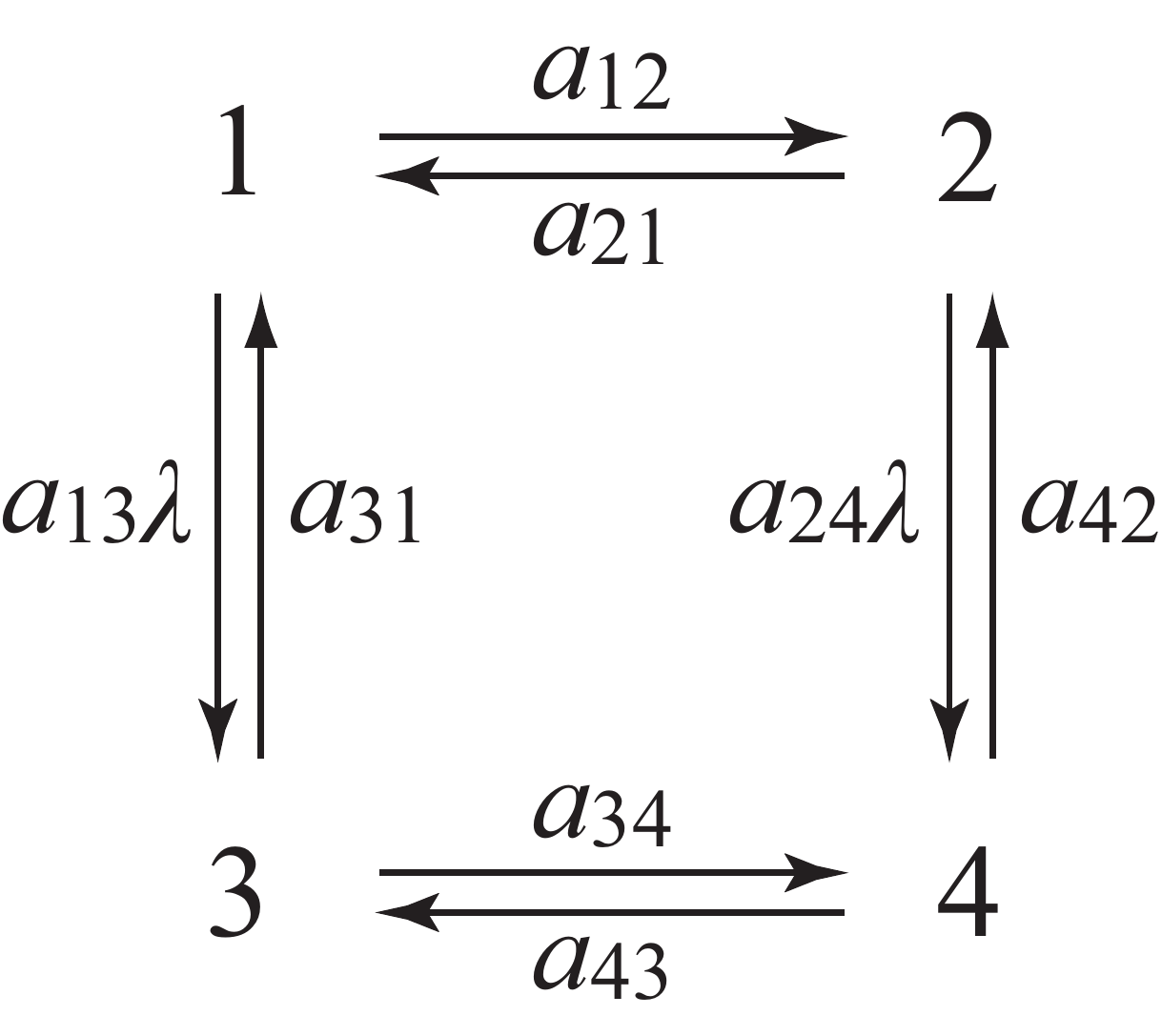}}
\caption{The Monod-Wyman-Changeux allosteric model.}\label{MWCmodel}
\end{center}
\end{figure}

By relabeling the state space $S$, we can always arrange matters so that $A = \{1,\cdots,|A|\}$ and $B = \{|A|+1,\cdots,|S|\}$. From now on, we take for granted that we have done this. Thus the transition rate matrix $Q(\lambda)$ can be represented as a block matrix
\begin{equation}
Q(\lambda) = \begin{pmatrix} Q_{AA}(\lambda) & Q_{AB}(\lambda) \\ Q_{BA}(\lambda) & Q_{BB}(\lambda) \end{pmatrix}.
\end{equation}
Let $Q = \lim_{\lambda\rightarrow\infty}Q(\lambda)$. The matrix $Q$ can be also represented as a block matrix
\begin{equation}
Q  = \begin{pmatrix} Q_{AA} & Q_{AB} \\ Q_{BA} & Q_{BB} \end{pmatrix}.
\end{equation}
Note that some elements of the matrix $Q$ may be $\infty$ or $-\infty$. According to the definition of the slow state space $B$, the elements of both matrices $Q_{BA}$ and $Q_{BB}$ are all finite.

In order to study the limit behavior of the Markov chain $X$ as $\lambda\rightarrow\infty$, we need the help of the jump chain \cite{norris1998markov}, also called the imbedded chain. Let $\xi = \{\xi_n:n\geq 0\}$ be the jump chain of $X$ with transition probability matrix $\Omega(\lambda) = (\omega_{ij}(\lambda))$ where $\omega_{ij}(\lambda) = q_{ij}(\lambda)/q_i(\lambda)$ for any pair of states $i\neq j$, and $\omega_{ii}(\lambda) = 0$ for any state $i$.  We also represent $\Omega(\lambda)$ as a block matrix
\begin{equation}
\Omega(\lambda) = \begin{pmatrix} \Omega_{AA}(\lambda) & \Omega_{AB}(\lambda) \\ \Omega_{BA}(\lambda) & \Omega_{BB}(\lambda) \end{pmatrix}.
\end{equation}
We further assume that $\lim_{\lambda\rightarrow\infty}\omega_{ij}(\lambda)$ exists for any pair of states $i\neq j$. Let $\Omega = \lim_{\lambda\rightarrow\infty}\Omega(\lambda)$. We also represent $\Omega = (\omega_{ij})$ as a block matrix
\begin{equation}
\Omega = \begin{pmatrix} \Omega_{AA} & \Omega_{AB} \\ \Omega_{BA} & \Omega_{BB} \end{pmatrix}.
\end{equation}
Since $\Omega(\lambda)$ is a stochastic matrix for each $\lambda>0$, $\Omega$ is also a stochastic matrix. Let $\eta=\{\eta_n:n\geq 0\}$ be a discrete-time Markov chain with transition probability matrix $\Omega$.

\section{Reduction of the Markov chain over finite time intervals}
In the following discussion, we shall study the limit behavior of the Markov chain $X$ as $\lambda\rightarrow\infty$.
Let $B$ be the slow state space. Let
\begin{equation}
T_B = \inf\{n\geq 0:\eta_n\in B\}
\end{equation}
be the first-passage time of $B$ for the discrete-time Markov chain $\eta$.

\begin{lemma}\label{invertible}
Assume that $P_i(T_B<\infty) = 1$ for any $i\in A$. Then the matrix $I-\Omega_{AA}$ is invertible and
\begin{equation}
(I-\Omega_{AA})^{-1} = \sum_{n=0}^\infty \Omega_{AA}^n.
\end{equation}
\end{lemma}

\begin{proof}
Note that $\Omega_{AA}$ is a nonnegative matrix and the sum of the elements in each row of $\Omega_{AA}$ is less than or equal to 1. Since $P_i(T_B<\infty) = 1$ for any $i\in A$, any subset of $A$ is not a closed set of the Markov chain $\eta$. By the Perron-Frobenius theorem, the absolute values of all eigenvalues of $\Omega_{AA}$ are less than 1. This shows that $I-\Omega_{AA}$ is invertible and $(I-\Omega_{AA})^{-1} = \sum_{n=0}^\infty \Omega_{AA}^n$.
\end{proof}

\begin{lemma}
Assume that $P_i(T_B<\infty) = 1$ for any $i\in A$. Then the matrix
\begin{equation}\label{induced}
\Gamma = Q_{BB}+Q_{BA}(I-\Omega_{AA})^{-1}\Omega_{AB}
\end{equation}
is a transition rate matrix on the slow state space $B$.
\end{lemma}

\begin{proof}
Let $\Gamma = (\gamma_{ij})$. By Lemma \ref{invertible}, $I-\Omega_{AA}$ is invertible and $(I-\Omega_{AA})^{-1} = \sum_{n=0}^\infty(\Omega_{AA})^n$.
Let
\begin{equation}
f_{ij}^{(n)} = [(\Omega_{AA})^n]_{ij},
\end{equation}
where $[M]_{ij}$ represents the element of the matrix $M$ in the $i$-th row and the $j$-th column. It is easy to see that $f_{ij}^{(n)}\geq 0$. Thus for any $i,j\in B$ and $i\neq j$,
\begin{equation}\label{gamma}
\gamma_{ij} = q_{ij}+\sum_{n=0}^\infty\sum_{l,m\in A}q_{il}f_{lm}^{(n)}\omega_{mj} \geq 0.
\end{equation}
We still need to prove that the sum of elements in each row of $\Gamma$ is 0. To this end, denote by $1$ the column vector whose elements are all 1. We only need to prove that $\Gamma1 = 0$. In fact,
\begin{equation}
\begin{split}
\Gamma1 &= Q_{BB}1 + Q_{BA}(I-\Omega_{AA})^{-1}\Omega_{AB}1 \\
&= Q_{BB}1 + Q_{BA}(I-\Omega_{AA})^{-1}(1-\Omega_{AA}1) \\
&= Q_{BB}1 + Q_{BA}1 = 0.
\end{split}
\end{equation}
This completes the proof of this lemma.
\end{proof}

In the following discussion, we shall always assume that $P_i(T_B<\infty) = 1$ for any $i\in A$.
\begin{definition}
The Markov chain $Y = \{Y_t:t\geq 0\}$ on the slow state space $B$ with transition rate matrix $\Gamma = Q_{BB}+Q_{BA}(I-\Omega_{AA})^{-1}\Omega_{AB}$ is called the reduced chain.
\end{definition}

Let $\Gamma=(\gamma_{ij})$ be the transition rate matrix of the reduced chain $Y$. Consistent with standard notations \cite{norris1998markov}, we set $\gamma_i = -\gamma_{ii}$. It is easy to see that
\begin{equation}
\gamma_i = q_i-\sum_{n=0}^\infty\sum_{l,m\in A}q_{il}f_{lm}^{(n)}\omega_{mi}.
\end{equation}
The following inequalities are important for the estimation of exponential random variables.
\begin{lemma}\label{inequality}
$(1)$ For any $x\geq 0$, we have $e^{-x}\geq 1-x$. \\
$(2)$ For any complex number $x$ satisfying $|x|\leq 1$, we have $|e^{-x}-(1-x)| \leq |x|^2$. \\
$(3)$ For any $x>0$, we have $(1+1/x)^x \leq e$. \\
$(4)$ For any $0\leq x<1$, we have $-\log(1-x)\geq x$.
\end{lemma}

\begin{proof}
The proofs of (1), (3), and (4) are straightforward. The proof of (2) can be found in \cite{durrett2010probability}. We omit these proofs.
\end{proof}

\begin{lemma}\label{timeclose}
For any $t\geq0$, $h>0$, and $i,j\in S$,
\begin{equation}
|P^\lambda_i(X_{t+h}=j) - P^\lambda_i(X_t=j)| \leq 1-e^{-q_i(\lambda)h},
\end{equation}
where $P^\lambda_i$ denotes the probability measure under transition rate matrix $Q(\lambda)$ and initial state $i$. Similarly, for any $t\geq0$, $h>0$, and $i,j\in B$,
\begin{equation}
|P_i(Y_{t+h}=j) - P_i(Y_t=j)| \leq 1-e^{-\gamma_ih}.
\end{equation}
\end{lemma}

\begin{proof}
The proof of this lemma follows easily from the semigroup property of Markov chains. The detailed proof of this lemma can be found in \cite{norris1998markov}.
\end{proof}

The following lemma plays a key role in obtaining the asymptotic behavior of the Markov chain $X$. In order not to interrupt things, we defer the proof of this lemma to the final section of this paper. For convenience, we define a constant $M$ as
\begin{equation}
M = \max_{j\in B}{q_j}+\max_{j\in B}{\gamma_j}+1.
\end{equation}
\begin{lemma}\label{complicated}
Assume that there exists $\eta > 0$, such that for any $i,j\in B$,
\begin{equation}
P^\lambda_i(X_t=j)-P_i(Y_t=j) \geq -\eta.
\end{equation}
Then for any $h<1/M$ and $\epsilon>0$, there exists $\lambda(h,\epsilon)>0$, such that for any $\lambda>\lambda(h,\epsilon)$ and $i,j\in B$,
\begin{equation}
P^\lambda_i(X_{t+h}=j)-P_i(Y_{t+h}=j) \geq -\left((1+|B|Mh)\eta+4M^2h^2+\epsilon\right).
\end{equation}
\end{lemma}

We are now in a position to state the main result of this section.
\begin{theorem}\label{stronglocal}
For any $T>0$ and $i,j\in B$,
\begin{equation}
\lim_{\lambda\rightarrow\infty}\sup_{0\leq t\leq T}\left|P^\lambda_i(X_t=j)-P_i(Y_t=j)\right| = 0.
\end{equation}
\end{theorem}

\begin{proof}
In the following proof, we fix $T>0$ and $i,j\in B$. For any $0<\epsilon<1-1/e$, let $h=\epsilon/M$. By Lemma \ref{inequality}, we have
\begin{equation}
h \leq \frac{-\log(1-\epsilon)}{M} < \frac{1}{M}.
\end{equation}
When $\lambda$ is sufficiently large, we have $q_i(\lambda)\leq M$ and $\gamma_i\leq M$. This implies that
\begin{equation}
1-e^{-q_i(\lambda)h}\leq\epsilon,\;\;\;1-e^{-\gamma_ih}\leq\epsilon.
\end{equation}
For any $0\leq t\leq T$, let $n = [t/h]$. It is easy to see that $0\leq t-nh \leq h$. By Lemma \ref{timeclose}, we have
\begin{equation}\label{discrete1}
\begin{split}
\left|P^\lambda_i(X_t=j)-P^\lambda_i(X_{nh}=j)\right| &\leq 1-e^{-q_i(\lambda)h}\leq\epsilon, \\
\left|P_i(Y_t=j)-P_i(Y_{nh}=j)\right| &\leq 1-e^{-\gamma_ih}\leq\epsilon.
\end{split}
\end{equation}
Note that $P^\lambda_i(X_0=j)-P_i(Y_0=j)=\delta_{ij}-\delta_{ij}=0$. Taking $\eta = 0$ and $t = 0$ in Lemma \ref{complicated}, we see that there exists $\lambda(\epsilon)>0$, such that for any $\lambda>\lambda(\epsilon)$,
\begin{equation}
P^\lambda_i(X_h=j)-P_i(Y_h=j) \geq -(4M^2h^2+\epsilon^2) = -5\epsilon^2.
\end{equation}
Using Lemma \ref{complicated} repeatedly, we see that for any $\lambda>\lambda(\epsilon)$,
\begin{equation}
\begin{split}
P^\lambda_i(X_{nh}=j)-P_i(Y_{nh}=j) &\geq -\left\{1+(1+|B|Mh)+\cdots+(1+|B|Mh)^{n-1}\right\}5\epsilon^2 \\
&= -\frac{(1+|B|Mh)^n-1}{|B|Mh}5\epsilon^2.
\end{split}
\end{equation}
Since $t\leq T$, we have $n\leq T/h$. By Lemma \ref{inequality}, it follows that
\begin{equation}
\begin{split}
\frac{(1+|B|Mh)^n-1}{|B|Mh}5\epsilon^2 &\leq \frac{(1+|B|Mh)^\frac{T}{h}-1}{|B|Mh}5\epsilon^2
\leq \frac{(1+|B|Mh)^\frac{|B|MT}{|B|Mh}-1}{|B|Mh}5\epsilon^2 \\
&\leq \frac{e^{|B|MT}}{|B|\epsilon}5\epsilon^2 = M_0\epsilon,
\end{split}
\end{equation}
where $M_0 = 5e^{|B|MT}/|B|$ is a positive constant. Thus we have
\begin{equation}
P^\lambda_i(X_{nh}=j)-P_i(Y_{nh}=j) \geq -M_0\epsilon.
\end{equation}
This implies that
\begin{equation}
\begin{split}
P^\lambda_i(X_{nh}=j) &\leq 1 - \sum_{k\in B\atop k\neq j}P^\lambda_i(X_{nh}=k) \leq 1-\sum_{k\in B\atop k\neq j}(P_i(Y_{nh}=k)-M_0\epsilon) \\
&\leq P_i(Y_{nh}=j) + |B|M_0\epsilon.
\end{split}
\end{equation}
Thus we obtain that
\begin{equation}\label{discrete2}
|P^\lambda_i(X_{nh}=j)-P_i(Y_{nh}=j)| \leq |B|M_0\epsilon.
\end{equation}
In view of \eqref{discrete1} and \eqref{discrete2}, when $\lambda$ is sufficiently large, for any $0\leq t\leq T$,
\begin{eqnarray*}
&& |P^\lambda_i(X_t=j)-P_i(Y_t=j)| \\
&\leq& |P^\lambda_i(X_t=j)-P^\lambda_i(X_{nh}=j)|+|P^\lambda_i(X_{nh}=j)-P_i(Y_{nh}=j)| \\
&& +|P_i(Y_{nh}=j)-P_i(Y_t=j)| \\
&\leq& \epsilon+|B|M_0\epsilon+\epsilon = (|B|M_0+2)\epsilon.
\end{eqnarray*}
This implies the result of this theorem.
\end{proof}

\begin{definition}
Let $\mu$ and $\nu$ be two probability measures on the state space $S$. Then the total variation distance between $\mu$ and $\nu$ is defined as
\begin{equation}
d_{TV}(\mu,\nu) = \frac{1}{2}\sum_{i\in S}|\mu_i-\nu_i|.
\end{equation}
\end{definition}
It is straightforward to check that the set of all probability measures on $S$ is a complete metric space under the total variation distance. The following result is a direct corollary of Theorem \ref{stronglocal}.

\begin{corollary}\label{dtvlocal}
Let $\pi$ be a probability distribution concentrated on the slow state space $B$. Then for any $T>0$,
\begin{equation}
\lim_{\lambda\rightarrow\infty}\sup_{0\leq t\leq T}d_{TV}(P^\lambda_{\pi}(X_t\in\cdot),P_{\pi}(Y_t\in\cdot)) = 0.
\end{equation}
\end{corollary}

\begin{proof}
For any $i\in B$ and $j\notin B$,
\begin{equation}
\begin{split}
& |P^\lambda_i(X_t=j)-P_i(Y_t=j)| = P^\lambda_i(X_t=j) \leq 1-\sum_{k\in B}P^\lambda_i(X_t=k) \\
&= \sum_{k\in B}(P_i(Y_t=k)-P^\lambda_i(X_t=k)) \leq \sum_{k\in B}|P^\lambda_i(X_t=k)-P_i(Y_t=k)|.
\end{split}
\end{equation}
This shows that
\begin{equation}
\sup_{0\leq t\leq T}|P^\lambda_i(X_t=j)-P_i(Y_t=j)| \leq \sum_{k\in B}\sup_{0\leq t\leq T}|P^\lambda_i(X_t=k)-P_i(Y_t=k)|.
\end{equation}
By Theorem \ref{stronglocal}, it is easy to that for any $i\in B$ and $j\in S$,
\begin{equation}\label{cor1}
\lim_{\lambda\rightarrow\infty}\sup_{0\leq t\leq T}|P^\lambda_i(X_t=j)-P_i(Y_t=j)| =0.
\end{equation}
Moreover, we have
\begin{equation}
\begin{split}
& d_{TV}(P^\lambda_{\pi}(X_t\in\cdot),P_{\pi}(Y_t\in\cdot)) = \frac{1}{2}\sum_{j\in S}|P^\lambda_{\pi}(X_t=j)-P_{\pi}(Y_t=j)| \\
&\leq \frac{1}{2}\sum_{i\in B}\sum_{j\in S}\pi_i|P^\lambda_i(X_t=j)-P_i(Y_t=j)|.
\end{split}
\end{equation}
Thus we obtain that
\begin{equation}\label{cor2}
\sup_{0\leq t\leq T}d_{TV}(P^\lambda_{\pi}(X_t\in\cdot),P_{\pi}(Y_t\in\cdot)) \leq \frac{1}{2}\sum_{i\in B}\sum_{j\in S}\pi_i\sup_{0\leq t\leq T}|P^\lambda_i(X_t=j)-P_i(Y_t=j)|.
\end{equation}
In view of \eqref{cor1} and \eqref{cor2}, we obtain the result of this corollary.
\end{proof}

The above corollary suggests that if the initial distribution of the original chain $X$ is concentrated on the slow state space $B$, then the distribution of the original chain $X$ will be very closed to that of the reduced chain $Y$ over any finite time interval when $\lambda$ is sufficiently large. One may ask whether the fixed time $T>0$ in Theorem \ref{stronglocal} can be replaced by infinity, that is, whether for any $i,j\in B$,
\begin{equation}\label{appglobal}
\lim_{\lambda\rightarrow\infty}\sup_{t\geq 0}\left|P^\lambda_i(X_t=j)-P_i(Y_t=j)\right| = 0.
\end{equation}
In general, the answer to this question is false. This phenomenon is quite similar to the continuous dependence on the initial value of the solution to an ordinary differential equation, where we can only prove that for any fixed time $T>0$, the solutions up to time $T$ are close to each other if the initial values are close enough. The next example shows that \eqref{appglobal} in general does not hold.
\begin{figure}[!htb]
\begin{center}
\centerline{\includegraphics[width=0.25\textwidth]{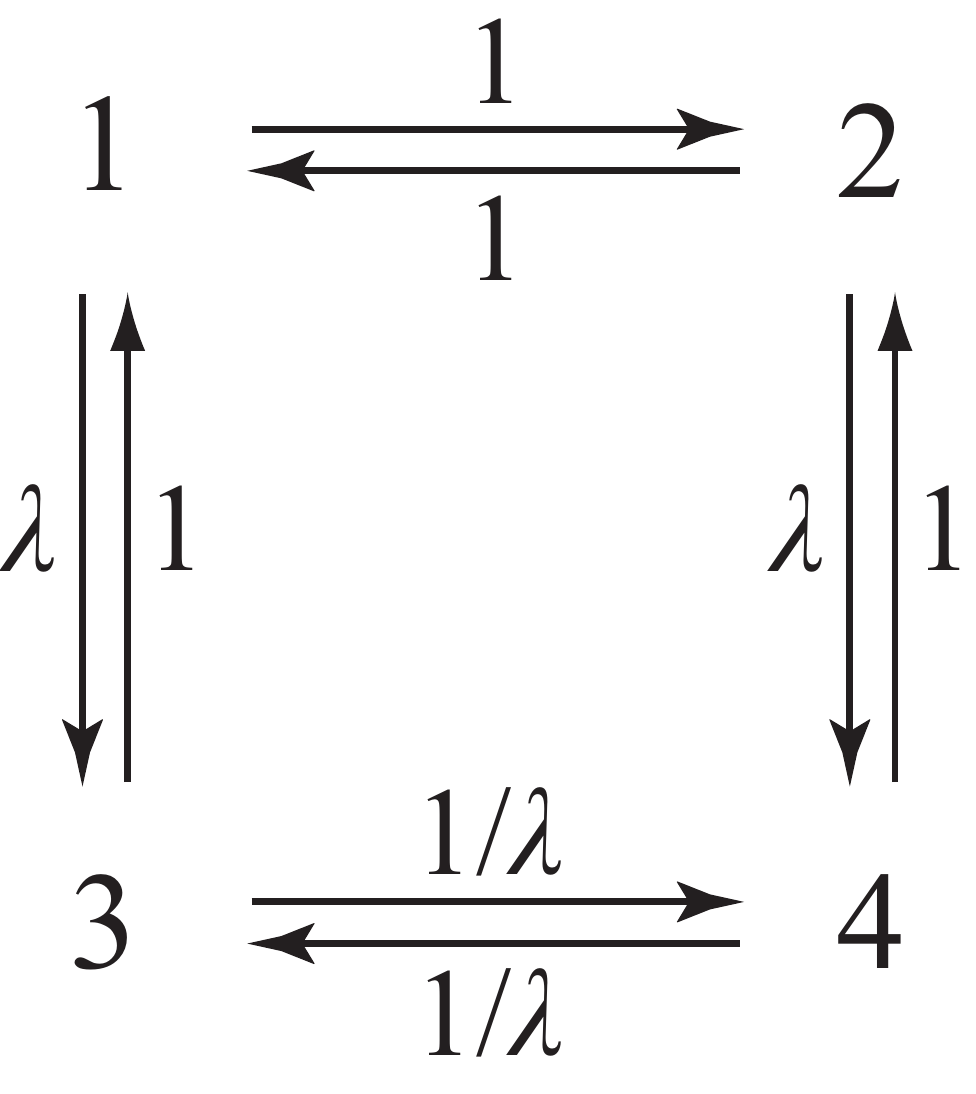}}
\caption{A counterexample showing that the fixed time $T$ in Theorem \ref{stronglocal} in general cannot be replaced by $\infty$.}\label{counterexample}
\end{center}
\end{figure}

\begin{example}\label{counter}
Consider the Markov chain $X$ illustrated in Figure \ref{counterexample}.
According to Definition \ref{states}, the fast state space is $A = \{1,2\}$ and the slow state space is $B = \{3,4\}$. In order to calculate the transition rate matrix of the reduced chain $Y$, we first need to obtain $Q(\lambda)$ and $\Omega(\lambda)$. It is easy to see that
\begin{equation}
Q(\lambda) =
\begin{pmatrix}
-(\lambda+1) & 1 & \lambda & 0 \\
1 & -(\lambda+1) & 0 & \lambda \\
1 & 0 & -(\frac{1}{\lambda}+1) & \frac{1}{\lambda} \\
0 & 1 & \frac{1}{\lambda} & -(\frac{1}{\lambda}+1) \\
 \end{pmatrix}.
\end{equation}
and
\begin{equation}
\Omega(\lambda) =
\begin{pmatrix}
0 & \frac{1}{\lambda+1} & \frac{\lambda}{\lambda+1} & 0 \\
\frac{1}{\lambda+1} & 0 & 0 & \frac{\lambda}{\lambda+1} \\
\frac{\lambda}{\lambda+1} & 0 & 0 & \frac{1}{\lambda+1} \\
0 & \frac{\lambda}{\lambda+1} & \frac{1}{\lambda+1} & 0 \\
 \end{pmatrix}.
\end{equation}
Thus $Q_{BB} = \lim_{\lambda\rightarrow\infty}Q_{BB}(\lambda) = -I$, $Q_{BA} = \lim_{\lambda\rightarrow\infty}Q_{BA}(\lambda) = I$, $\Omega_{AA} = \lim_{\lambda\rightarrow\infty}\Omega_{AA}(\lambda) = 0$, and $\Omega_{AB} = \lim_{\lambda\rightarrow\infty}\Omega_{AB}(\lambda) = I$. Thus the transition rate matrix $\Gamma$ of the reduced chain $Y$ is
\begin{equation}
\Gamma =  Q_{BB}+Q_{BA}(I-\Omega_{AA})^{-1}\Omega_{AB} = -I+I = 0.
\end{equation}
This shows that the reduced chain $Y$ is a constant process. Particularly, we have $P_3(Y_t = 4) = 0$ for any $t\geq 0$. On the other hand, it is easy to see that the invariant distribution $\mu(\lambda)$ of the original chain $X$ is
\begin{equation}
\mu(\lambda) = \left(\frac{1}{2(\lambda+1)},\frac{1}{2(\lambda+1)},\frac{\lambda}{2(\lambda+1)},\frac{\lambda}{2(\lambda+1)}\right).
\end{equation}
By the convergence theorem of irreducible Markov chains, we have
\begin{equation}
\lim_{t\rightarrow\infty}P^\lambda_3(X_t = 4) = \mu_4(\lambda) = \frac{\lambda}{2(\lambda+1)}.
\end{equation}
Thus we obtain that
\begin{equation}
\lim_{\lambda\rightarrow\infty}\lim_{t\rightarrow\infty}\left|P^\lambda_3(X_t = 4)- P_3(Y_t = 4)\right| = \frac{1}{2}.
\end{equation}
If \eqref{appglobal} holds, then we must have
\begin{equation}
\lim_{\lambda\rightarrow\infty}\lim_{t\rightarrow\infty}\left|P^\lambda_3(X_t = 4)- P_3(Y_t = 4)\right| = 0.
\end{equation}
This shows that \eqref{appglobal} is not satisfied in our current example. We can only prove that the distributions of the original chain $X$ and the reduced chain $Y$ are close to each other over any finite time interval when $\lambda$ is sufficiently large, but we fail to conclude that their distributions are close to each other over the whole time axis.
\end{example}

\section{Reduction of the Markov chain over the whole time axis}
In example \ref{counter}, the reduced chain $Y$ has more than one recurrent class. The fact that different recurrent classes cannot communicate with each other results in the strange phenomenon that for some $i,j\in B$,
\begin{equation}
\lim_{\lambda\rightarrow\infty}\lim_{t\rightarrow\infty}\left|P_i^\lambda(X_t = j)- P_i(Y_t = j)\right| \neq 0.
\end{equation}
This then raises a natural question: if the reduced chain $Y$ is irreducible, can we prove that for any $i,j\in B$,
\begin{equation}
\lim_{\lambda\rightarrow\infty}\sup_{t\geq 0}\left|P^\lambda_i(X_t=j)-P_i(Y_t=j)\right| = 0?
\end{equation}
Interestingly, the answer to this question is affirmative. To prove this fact, we first prove the following theorem, which shows that if the reduced chain $Y$ is irreducible, then the invariant distribution of the original chain $X$ will converge to that of the reduced chain $Y$ as $\lambda\rightarrow\infty$.

\begin{theorem}\label{invconverge}
Assume that the reduced chain $Y$ is irreducible. Let $\mu(\lambda)=(\mu_A(\lambda),\mu_B(\lambda))$ be the invariant distribution of the Markov chain $X$ and let $\mu_B$ be the invariant distribution of the reduced chain $Y$. Then $\lim_{\lambda\rightarrow\infty}\mu_B(\lambda) = \mu_B$ and $\lim_{\lambda\rightarrow\infty}\mu_A(\lambda) = 0$.
\end{theorem}

\begin{proof}
Let $M_A(\lambda) = \diag(q_1(\lambda),\cdots,q_{|A|}(\lambda))$ be the diagonal matrix whose diagonal elements are $q_1(\lambda),\cdots,q_{|A|}(\lambda)$, respectively. Note that $q_{ij}(\lambda) = q_i(\lambda)\omega_{ij}(\lambda)$ for any $i\neq j$ and $q_{ii}(\lambda) = -q_i(\lambda)$ for any $i\in S$. Thus for any $i,j\in S$,
\begin{equation}
q_{ij}(\lambda) = q_i(\lambda)(\omega_{ij}(\lambda)-\delta_{ij}).
\end{equation}
This implies that
\begin{equation}\label{omega}
Q_{AB}(\lambda) = M_A(\lambda)\Omega_{AB}(\lambda),\;\;\;
Q_{AA}(\lambda) = -M_A(\lambda)(I-\Omega_{AA}(\lambda)).
\end{equation}
Let $Q^c(\lambda)$ be the matrix obtained from $Q(\lambda)$ by changing all elements in the rightmost column of $Q(\lambda)$ to 1. Then $Q^c(\lambda)$ can be represented as the block matrix
\begin{equation}
Q^c(\lambda)  =
\begin{pmatrix}
Q_{AA}(\lambda) & Q^c_{AB}(\lambda) \\ Q_{BA}(\lambda) & Q^c_{BB}(\lambda)
\end{pmatrix},
\end{equation}
where $Q^c_{AB}(\lambda)$ and $Q^c_{BB}(\lambda)$ are matrices obtained from $Q_{AB}(\lambda)$ and $Q_{BB}(\lambda)$ by changing all elements in the rightmost column of $Q_{AB}(\lambda)$ and $Q_{BB}(\lambda)$ to 1, respectively. Since the Markov chain $X$ is irreducible, the matrix $Q^c(\lambda)$ is invertible. Note that $\mu(\lambda)Q^c(\lambda) = (0,\cdots,0,1)$. Thus we have
\begin{equation}
\mu(\lambda) = (0,\cdots,0,1)(Q^c(\lambda))^{-1}.
\end{equation}
By the formula of inversion of block matrices, we obtain that
\begin{equation}
\mu_B(\lambda) = (0,\cdots,0,1)(Q^c_{BB}(\lambda)-Q_{BA}(\lambda)Q_{AA}(\lambda)^{-1}Q^c_{AB}(\lambda))^{-1}.
\end{equation}
Similarly, let $\Gamma^c$ be the matrix obtained from $\Gamma$ by changing all elements in the rightmost column of $\Gamma$ to 1. Since the reduced chain $Y$ is assumed to be irreducible, the matrix $\Gamma^c$ is invertible. Note that $\mu_B\Gamma^c = (0,\cdots,0,1)$. Thus we have
\begin{equation}
\mu_B = (0,\cdots,0,1)(\Gamma^c)^{-1}.
\end{equation}
In order to prove that $\lim_{\lambda\rightarrow\infty}\mu_B(\lambda)=\mu_B$, it suffices to prove that
\begin{equation}
\lim_{\lambda\rightarrow\infty}Q^c_{BB}(\lambda)-Q_{BA}(\lambda)Q_{AA}(\lambda)^{-1}Q^c_{AB}(\lambda) = \Gamma^c.
\end{equation}
To this end, we only need to prove the following two equalities:
\begin{equation}
\begin{split}
&\lim_{\lambda\rightarrow\infty}Q_{BB}(\lambda)-Q_{BA}(\lambda)Q_{AA}(\lambda)^{-1}Q_{AB}(\lambda) = \Gamma, \\
&\lim_{\lambda\rightarrow\infty}1-Q_{BA}(\lambda)Q_{AA}(\lambda)^{-1}1 = 1.
\end{split}
\end{equation}
To prove the first equality, note that
\begin{equation}
\begin{split}
& \lim_{\lambda\rightarrow\infty}Q_{BB}(\lambda)-Q_{BA}(\lambda)Q_{AA}(\lambda)^{-1}Q_{AB}(\lambda) \\
&= \lim_{\lambda\rightarrow\infty}Q_{BB}(\lambda)+ Q_{BA}(\lambda)(I-\Omega_{AA}(\lambda))^{-1}M_A(\lambda)^{-1}M_A(\lambda)\Omega_{AB}(\lambda) \\
&= Q_{BB}+Q_{BA}(I-\Omega_{AA})^{-1}\Omega_{AB} = \Gamma.
\end{split}
\end{equation}
To prove the second equality, note that $\lim_{\lambda\rightarrow\infty}M_A(\lambda)^{-1} = 0$. Thus we obtain that
\begin{equation}
\lim_{\lambda\rightarrow\infty}1-Q_{BA}(\lambda)Q_{AA}(\lambda)^{-1}1 = \lim_{\lambda\rightarrow\infty}1+Q_{BA}(\lambda)(I-\Omega_{AA}(\lambda))^{-1}M_A(\lambda)^{-1}1 = 1.
\end{equation}
Therefore, we have proved that $\lim_{\lambda\rightarrow\infty}\mu_B(\lambda)=\mu_B$. Note that
\begin{equation}
\lim_{\lambda\rightarrow\infty}\mu_A(\lambda)1 = \lim_{\lambda\rightarrow\infty}1-\mu_B(\lambda)1 = 1-\mu_B1 = 0.
\end{equation}
This implies that $\lim_{\lambda\rightarrow\infty}\mu_A(\lambda)=0$.
\end{proof}

The next lemma will be proved in Section \ref{general}. We only state the result here.
\begin{lemma}\label{point}
Let $\pi=(\pi_A,\pi_B)$ be a probability distribution on the state space $S$. Then for any $h>0$ and $j\in S$,
\begin{equation}
\lim_{\lambda\rightarrow\infty}P^\lambda_\pi(X_h=j) = P_{\gamma(\pi)}(Y_h=j),
\end{equation}
where
\begin{equation}
\gamma(\pi) = \pi_B+\pi_A(I-\Omega_{AA})^{-1}\Omega_{AB}
\end{equation}
is a probability distribution on the slow state space $B$.
\end{lemma}

\begin{proof}
This lemma is a direct corollary of Theorem \ref{weaklocal} in Section \ref{general}.
\end{proof}

\begin{lemma}\label{concentrate}
For any $h>0$ and $\epsilon>0$, when $\lambda$ is sufficiently large,
\begin{equation}
\sup_{n\geq1}\sup_{i\in S}P^\lambda_i(X_{nh}\in A) \leq \epsilon.
\end{equation}
\end{lemma}

\begin{proof}
By Lemma \ref{point}, it is easy to see that
\begin{equation}
\lim_{\lambda\rightarrow\infty}\sup_{i\in S}P^\lambda_i(X_h\in A) = 0.
\end{equation}
Thus when $\lambda$ is sufficiently large,
\begin{equation}
\sup_{i\in S}P^\lambda_i(X_h\in A) \leq \epsilon.
\end{equation}
Thus for any $n\geq 1$ and $i\in S$,
\begin{equation}
P^\lambda_i(X_{nh}\in A) = \sum_{j\in S}P^\lambda_i(X_{(n-1)h}=j)P^\lambda_j(X_h\in A) \leq \epsilon\sum_{j\in S}P^\lambda_i(X_{(n-1)h}=j) = \epsilon.
\end{equation}
This completes the proof of this lemma.
\end{proof}

\begin{lemma}\label{iterationinv}
Assume that the reduced chain $Y$ is irreducible. Let $\mu(\lambda)=(\mu_1(\lambda),\cdots,\mu_{|S|}(\lambda))$ be the invariant distribution of the Markov chain $X$. Then for any $h>0$ and $\epsilon>0$, when $\lambda$ is sufficiently large, for any $n\geq 1$ and $i\in S$,
\begin{equation}
\sum_{j\in B}|P^\lambda_i(X_{(n+1)h}=j)-\mu_j(\lambda)| \leq (1-|B|\gamma(h))\sum_{j\in B}|P^\lambda_i(X_{nh}=j)-\mu_j(\lambda)|+\epsilon,
\end{equation}
where
\begin{equation}
\gamma(h) = \frac{1}{2}\min_{i,j\in B}P_i(Y_h=j)>0.
\end{equation}
\end{lemma}

\begin{proof}
Since the reduced chain $Y$ is irreducible, we have $P_i(Y_h=j)>0$ for any $i,j\in B$. This shows that $\gamma(h)>0$. By Lemma \ref{point}, for any $j\in B$,
\begin{equation}
\lim_{\lambda\rightarrow\infty}P^\lambda_\pi(X_h=j) = \sum_{i\in B}\lambda_i(\pi)P_i(Y_h=j) \geq \min_{i,j\in B}P_i(Y_h=j).
\end{equation}
Thus we obtain that
\begin{equation}
\lim_{\lambda\rightarrow\infty}\min_{i\in S\atop j\in B}P^\lambda_i(X_h=j) = \min_{i\in S\atop j\in B}\lim_{\lambda\rightarrow\infty}P^\lambda_i(X_h=j) \geq \min_{i,j\in B}P_i(Y_h=j) > \gamma(h).
\end{equation}
Thus when $\lambda$ is sufficiently large,
\begin{equation}\label{iteration1}
\min_{i\in S\atop j\in B}P^\lambda_i(X_h=j) > \gamma(h).
\end{equation}
By Lemma \ref{concentrate}, when $\lambda$ is sufficiently large,
\begin{equation}\label{iteration2}
\sup_{n\geq1}\sup_{i\in S}P^\lambda_i(X_{nh}\in A) \leq \frac{\epsilon}{2}.
\end{equation}
Combining \eqref{iteration1} and \eqref{iteration2}, we see that for any $n\geq 1$ and $i\in S$,
\begin{eqnarray*}
&& \sum_{j\in B}|P^\lambda_i(X_{(n+1)h}=j)-\mu_j(\lambda)| = \sum_{j\in B}|P^\lambda_i(X_{(n+1)h}=j)-P^\lambda_{\mu(\lambda)}(X_{(n+1)h}=j)| \\
&=& \sum_{j\in B}|\sum_{k\in S}P^\lambda_i(X_{nh}=k)P^\lambda_k(X_h=j)-\sum_{k\in S}P^\lambda_{\mu(\lambda)}(X_{nh}=k)P^\lambda_k(X_h=j)| \\
&=& \sum_{j\in B}|\sum_{k\in S}(P^\lambda_i(X_{nh}=k)-P^\lambda_{\mu(\lambda)}(X_{nh}=k))(P^\lambda_k(X_h=j)-\gamma(h))| \\
&\leq& \sum_{k\in S}\sum_{j\in B}|P^\lambda_i(X_{nh}=k)-P^\lambda_{\mu(\lambda)}(X_{nh}=k)|(P^\lambda_k(X_h=j)-\gamma(h)) \\
&=& \sum_{k\in S}|P^\lambda_i(X_{nh}=k)-P^\lambda_{\mu(\lambda)}(X_{nh}=k)|(P^\lambda_k(X_h\in B)-|B|\gamma(h)) \\
&\leq& (1-|B|\gamma(h))\sum_{k\in B}|P^\lambda_i(X_{nh}=k)-P^\lambda_{\mu(\lambda)}(X_{nh}=k)| + \sum_{k\in A}(P^\lambda_i(X_{nh}=k)+P^\lambda_{\mu(\lambda)}(X_{nh}=k)) \\
&=& (1-|B|\gamma(h))\sum_{j\in B}|P^\lambda_i(X_{nh}=j)-\mu_j(\lambda)| + P^\lambda_i(X_{nh}\in A)+P^\lambda_{\mu(\lambda)}(X_{nh}\in A) \\
&\leq& (1-|B|\gamma(h))\sum_{j\in B}|P^\lambda_i(X_{nh}=j)-\mu_j(\lambda)| + \epsilon.
\end{eqnarray*}
This completes the proof of this lemma.
\end{proof}

\begin{lemma}\label{speed}
Assume that the reduced chain $Y$ is irreducible. Let $\mu(\lambda)=(\mu_1(\lambda),\cdots,\mu_{|S|}(\lambda))$ be the invariant distribution of the Markov chain $X$. Then for any $h>0$ and $\epsilon>0$, when $\lambda$ is sufficiently large, for any $n\geq 1$ and $i\in S$,
\begin{equation}
\sum_{j\in B}|P^\lambda_i(X_{nh}=j)-\mu_j(\lambda)| \leq 2(1-|B|\gamma(h))^{n-1}+\epsilon.
\end{equation}
\end{lemma}

\begin{proof}
Let $a^i_n = \sum_{j\in B}|P^\lambda_i(X_{nh}=j)-\mu_j(\lambda)|$. By Lemma \ref{iterationinv}, for any $h>0$ and $\epsilon>0$, when $\lambda$ is sufficiently large, for any $n\geq 2$ and $i\in S$,
\begin{equation}
a^i_n \leq (1-|B|\gamma(h))a^i_{n-1}+|B|\gamma(h)\epsilon.
\end{equation}
Using the above relation repeatedly, we obtain that
\begin{equation}
a^i_n \leq (1-|B|\gamma(h))^{n-1}a^i_1+\left((1-|B|\gamma(h))^{n-2}+\cdots+1\right)|B|\gamma(h)\epsilon.
\end{equation}
Note that
\begin{equation}
a^i_1 \leq \sum_{j\in B}(P^\lambda_i(X_h=j)+\mu_j(\lambda)) \leq 2.
\end{equation}
Thus we obtain that
\begin{equation}
a^i_n \leq 2(1-|B|\gamma(h))^{n-1}+\frac{1}{|B|\gamma(h)}|B|\gamma(h)\epsilon = 2(1-|B|\gamma(h))^{n-1}+\epsilon.
\end{equation}
This shows that the lemma holds for $n\geq 2$. For $n=1$, it is easy to check that the lemma also holds.
\end{proof}

We are now in a position to state the main result of this section.
\begin{theorem}\label{strongglobal}
Assume that the reduced chain $Y$ is irreducible. Then for any $i,j\in B$,
\begin{equation}
\lim_{\lambda\rightarrow\infty}\sup_{t\geq 0}\left|P^\lambda_i(X_t=j)-P_i(Y_t=j)\right| = 0.
\end{equation}
\end{theorem}

\begin{proof}
In the following proof, we fix $i,j\in B$. For any $0<\epsilon<1$, let $h=\epsilon/M$. When $\lambda$ is sufficiently large, we have $q_i(\lambda)h \leq Mh = \epsilon \leq -\log(1-\epsilon)$. This shows that
\begin{equation}
1-e^{-q_i(\lambda)h} \leq \epsilon.
\end{equation}
Let
\begin{equation}
T_0 = \left(\frac{\log\epsilon}{\log(1-|B|\gamma(h))}+1\right)h.
\end{equation}
Thus for any $t\geq T_0$,
\begin{equation}
\frac{t}{h} \geq \frac{T_0}{h} = \frac{\log\epsilon}{\log(1-|B|\gamma(h))}+1.
\end{equation}
Choose $n\geq 1$ such that $(n-1)h\leq t\leq nh$. Then we have
\begin{equation}
(1-|B|\gamma(h))^{n-1} \leq \epsilon.
\end{equation}
By Lemma \ref{timeclose} and Lemma \ref{speed}, it follows that
\begin{equation}
\begin{split}
|P^\lambda_i(X_t=j)-\mu_j(\lambda)| &\leq |P^\lambda_i(X_t=j)-P^\lambda_i(X_{nh}=j)| + |P^\lambda_i(X_{nh}=j)-\mu_j(\lambda)| \\
&\leq 1-e^{q_i(\lambda)h} + 2(1-|B|\gamma(h))^{n-1} + \epsilon \leq 4\epsilon.
\end{split}
\end{equation}
Since the reduced chain $Y$ is irreducible, $Y$ has a unique invariant distribution $\mu_B=(\mu_{|A|+1},\cdots,\mu_{|S|})$. By the convergence theorem of irreducible Markov chains, we can choose $T\geq T_0$, such that for any $t\geq T$,
\begin{equation}
|P_i(Y_t=j)-\mu_j| \leq \epsilon.
\end{equation}
By Theorem \ref{invconverge}, when $\lambda$ is sufficiently large,
\begin{equation}
|\mu_j(\lambda)-\mu_j| \leq \epsilon.
\end{equation}
Thus when $\lambda$ is sufficiently large, for any $t\geq T$,
\begin{equation}\label{largetime}
\begin{split}
|P^\lambda_i(X_t=j)-P_i(Y_t=j)| &\leq |P^\lambda_i(X_t=j)-\mu_j(\lambda)| + |\mu_j(\lambda)-\mu_j| + |\mu_j-P_i(Y_t=j)| \\
&\leq 4\epsilon + \epsilon + \epsilon = 6\epsilon.
\end{split}
\end{equation}
By Theorem \ref{stronglocal}, when $\lambda$ is sufficiently large, for any $0\leq t\leq T$,
\begin{equation}\label{smalltime}
|P^\lambda_i(X_t=j)-P_i(Y_t=j)| \leq \epsilon.
\end{equation}
Combining \eqref{largetime} and \eqref{smalltime}, we complete the proof of this theorem.
\end{proof}

The next result is a direct corollary of Theorem \ref{strongglobal}.
\begin{corollary}\label{dtvglobal}
Assume that the reduced chain $Y$ is irreducible. Then for any probability distribution $\pi$ concentrated on the slow state space $B$,
\begin{equation}
\lim_{\lambda\rightarrow\infty}\sup_{t\geq 0}d_{TV}(P^\lambda_\pi(X_t\in\cdot),P_\pi(Y_t\in\cdot)) = 0.
\end{equation}
\end{corollary}

\begin{proof}
The proof of this corollary is totally the same as that of Corollary \ref{dtvlocal}.
\end{proof}

The above corollary shows that if the initial distribution of the original chain $X$ is concentrated on the slow state space $B$ and if the reduced chain $Y$ is irreducible, then the distribution of the original chain $X$ will be very close to that of the reduced chain $Y$ over the whole time axis when $\lambda$ is sufficiently large.

The reader may ask when the reduced chain $Y$ is irreducible. The following proposition gives a simple sufficient condition for the reduced chain $Y$ being irreducible.
\begin{proposition}
Assume that for any $i,j\in B$ and $i\neq j$, there exists $i_0,i_1,\cdots,i_n\in B$ with $i_0=i$ and $i_n=j$ such that $q_{i_0i_1}\cdots q_{i_{n-1}i_n}>0$. Then the reduced chain $Y$ is irreducible.
\end{proposition}

\begin{proof}
In view of \eqref{gamma}, it follows that $\gamma_{ij}\geq q_{ij}$ for any $i,j\in B$ and $i\neq j$. Therefore, for any $i,j\in B$ and $i\neq j$, there exists $i_0,i_1,\cdots,i_n\in B$ with $i_0=i$ and $i_n=j$ such that $\gamma_{i_0i_1}\cdots\gamma_{i_{n-1}i_n}>0$. This shows that the reduced chain $Y$ is irreducible.
\end{proof}

\section{Reduction of the Markov chain under general initial distributions}\label{general}
We have seen that if the initial distribution of the original chain $X$ is concentrated on the slow state space $B$, then the distributions of the original chain $X$ and the reduced chain $Y$ are close to each other when $\lambda$ is sufficiently large. However, what is the case if the initial distribution of $X$ is not concentrated on $B$? In this section, we shall prove that, although the initial distribution may not be concentrated on $B$, the distribution of $X$ will be ``almost" concentrated on $B$ after a very short time when $\lambda$ is sufficiently large. This fact implies the main result of this section, which shows that when the initial distribution is not concentrated on $B$, the distribution of the original chain $X$ will be very close to that of the reduced chain $Y$ after an arbitrarily small time $h>0$ when $\lambda$ is sufficiently large.

Let $B$ be the slow state space. Let
\begin{equation}
\tau_B = \inf\{t\geq 0:X_t\in B\}
\end{equation}
be the first-passage time of $B$ for the Markov chain $X$. In the previous discussion, we have defined the first-passage time of $B$ for the discrete-time Markov chain $\eta$ as
\begin{equation}
T_B = \inf\{n\geq 0:~\eta_n\in B\}.
\end{equation}
Recall that we always assume that $P_i(T_B<\infty) = 1$ for any $i\in A$.
\begin{lemma}\label{TB}
For any $i\in A$, we have
\begin{equation}
P_i(T_B<\infty) = \sum_{n=0}^\infty\sum_{k\in A}\sum_{j\in B}f_{ik}^{(n)}\omega_{kj}.
\end{equation}
where $f_{ik}^{(n)} = [(\Omega_{AA})^n]_{ik}$.
\end{lemma}

\begin{proof}
For any $i\in A$, it is easy to see that
\begin{equation}
\begin{split}
P_i(T_B<\infty) &= \sum_{n=0}^\infty P_i(\eta_0,\cdots,\eta_n\in A,\eta_{n+1}\in B) \\
&= \sum_{n=0}^\infty\sum_{k\in A}\sum_{j\in B}P_i(\eta_0,\cdots,\eta_n\in A,\eta_n=k,\eta_{n+1}=j) \\
&= \sum_{n=0}^\infty\sum_{k\in A}\sum_{j\in B}f_{ik}^{(n)}\omega_{kj},
\end{split}
\end{equation}
where we have used the fact that $P_i(\eta_0,\cdots,\eta_n\in A,\eta_n=k) = [(\Omega_{AA})^n]_{ik}$.
\end{proof}

\begin{lemma}\label{prep}
Let $\pi$ be a probability distribution on the state space $S$. Then for any $\lambda > 0$ and $\epsilon > 0$, there exists $t(\lambda,\epsilon) > 0$ such that $\lim_{\lambda\rightarrow\infty}t(\lambda,\epsilon) = 0$ and when $\lambda$ is sufficiently large,
\begin{equation}
P^\lambda_\pi(\tau_B\leq t(\lambda,\epsilon))\geq 1-\epsilon.
\end{equation}
\end{lemma}

\begin{proof}
Let $S_1,S_2,\cdots$ be the holding times of the Markov chain $X$ and let $J_n = \sum_{i=1}^nS_i$. Then $J_1,J_2,\cdots$ are the jump times of the Markov chain $X$. For any $i\in A$ and $t > 0$,
\begin{equation}\label{tauB}
\begin{split}
& P^\lambda_i(\tau_B\leq t)
= \sum_{n=0}^\infty P^\lambda_i(J_{n+1}\leq t,\xi_1,\cdots,\xi_n\in A,\xi_{n+1}\in B) \\
&= \sum_{n=0}^\infty\sum_{k_1,\cdots,k_n\in A}\sum_{j\in B}P^\lambda_i(J_{n+1}\leq t,\xi_1=k_1,\cdots,\xi_n=k_n,\xi_{n+1}=j) \\
&= \sum_{n=0}^\infty\sum_{k_1,\cdots,k_n\in A}\sum_{j\in B}P^\lambda_i(E_{i,k_1,\cdots,k_n}\leq t)\omega_{ik_1}(\lambda)\cdots\omega_{k_nj}(\lambda),
\end{split}
\end{equation}
where $E_{i,k_1,\cdots,k_n}$ is the sum of independent exponential random variables with parameters $q_i(\lambda),q_{k_1}(\lambda)$, $\cdots,q_{k_n}(\lambda)$, respectively. Recall that we have assumed that $P_i(T_B<\infty) = 1$ for any $i\in A$. By Lemma \ref{TB}, it follows that
\begin{equation}
\sum_{n=0}^\infty\sum_{k\in A}\sum_{j\in B}f_{ik}^{(n)}\omega_{kj} = 1.
\end{equation}
Thus we can choose a sufficiently large $N$, such that for any $i\in A$,
\begin{equation}\label{noI}
\sum_{n=0}^N\sum_{k\in A}\sum_{j\in B}f_{ik}^{(n)}\omega_{kj}\geq 1-\epsilon.
\end{equation}
We further set
\begin{equation}
t(\lambda,\epsilon) = \sup_{0\leq n\leq N\atop i,k_1,\cdots,k_n\in A}\left\{t>0:P^\lambda_i(E_{i,k_1,\cdots,k_n}\leq t)= 1-\epsilon\right\}.
\end{equation}
It is easy to see that $t(\lambda,\epsilon)>0$. Since $i,k_1,\cdots,k_n\in A$, by Slutsky's theorem, $E_{i,k_1,\cdots,k_n}$  converges in distribution to 0 as $\lambda\rightarrow\infty$. Thus for any $h>0$, when $\lambda$ is sufficiently large, for any $0\leq n\leq N$ and $i,k_1,\cdots,k_n\in A$, we have $P^\lambda_i(E_{i,k_1,\cdots,k_n}\leq h)> 1-\epsilon$. By the definition of $t(\lambda,\epsilon)$, it is easy to see that $t(\lambda,\epsilon)\leq h$. This clearly shows that $\lim_{\lambda\rightarrow\infty}t(\lambda,\epsilon) = 0$. In view of \eqref{tauB}, we obtain that
\begin{equation}
\begin{split}
P^\lambda_i(\tau_B\leq t(\lambda,\epsilon))
&\geq \sum_{n=0}^N\sum_{k_1,\cdots,k_n\in A}\sum_{j\in B}P^\lambda_i(E_{i,k_1,\cdots,k_n}\leq t(\lambda,\epsilon))\omega_{ik_1}(\lambda)\cdots\omega_{k_nj}(\lambda) \\
&\geq (1-\epsilon)\sum_{n=0}^N\sum_{k_1,\cdots,k_n\in A}\sum_{j\in B}\omega_{ik_1}(\lambda)\cdots\omega_{k_nj}(\lambda) \\
&= (1-\epsilon)\sum_{n=0}^N\sum_{k\in A}\sum_{j\in B}f_{ik}^{(n)}(\lambda)\omega_{kj}(\lambda),
\end{split}
\end{equation}
where $f_{ik}^{(n)}(\lambda) = [(\Omega_{AA}(\lambda))^n]_{ik}$. In view of \eqref{noI}, when $\lambda$ is sufficiently large, for any $i\in A$,
\begin{equation}
\sum_{n=0}^N\sum_{k\in A}\sum_{j\in B}f_{ik}^{(n)}(\lambda)\omega_{kj}(\lambda) > 1-2\epsilon.
\end{equation}
Thus we have
\begin{equation}
P^\lambda_i(\tau_B\leq t(\lambda,\epsilon)) \geq (1-\epsilon)(1-2\epsilon) \geq 1-3\epsilon.
\end{equation}
Since the above equation holds for any $i\in A$, we finally obtain that
\begin{equation}
\begin{split}
& P^\lambda_\pi(\tau_B\leq t(\lambda,\epsilon))
= P^\lambda_\pi(\tau_B=0) + P^\lambda_\pi(0<\tau_B\leq t(\lambda,\epsilon)) \\
&= \sum_{i\in B}\pi_i + \sum_{i\in A}\pi_iP^\lambda_i(\tau_B\leq t(\lambda,\epsilon))
\geq \sum_{i\in B}\pi_i + \sum_{i\in A}\pi_i(1-3\epsilon)
\geq 1-3\epsilon.
\end{split}
\end{equation}
This implies the result of this lemma.
\end{proof}

The following theorem, which is interesting in its own right, is a preparation theorem for the main result of this section.
\begin{theorem}\label{shortterm1}
Let $\pi$ be a probability distribution on the state space $S$. Then for any $\lambda > 0$ and $\epsilon > 0$, there exists $t(\lambda,\epsilon) > 0$ such that $\lim_{\lambda\rightarrow\infty}t(\lambda,\epsilon) = 0$ and when $\lambda$ is sufficiently large,
\begin{equation}
d_{TV}(P^\lambda_\pi(X_{t(\lambda,\epsilon)}\in\cdot),P^\lambda_\pi(X_{\tau_B}\in\cdot)) \leq \epsilon.
\end{equation}
\end{theorem}

\begin{proof}
Choose $t(\lambda,\epsilon)$ as in Lemma \ref{prep}. Note that for any $j\in B$, $X_{t(\lambda,\epsilon)}=j$ implies $\tau_B\leq t(\lambda,\epsilon)$. Thus we obtain that
\begin{equation}\label{p1}
\begin{split}
& P^\lambda_\pi(X_{t(\lambda,\epsilon)}=j)
\leq P^\lambda_\pi(X_{t(\lambda,\epsilon)}=j,\tau_B\leq t(\lambda,\epsilon)) \\
&= \sum_{k\in B}P^\lambda_\pi(X_{\tau_B}=k,X_{t(\lambda,\epsilon)}=j,\tau_B\leq t(\lambda,\epsilon)) \\
&\leq \sum_{k\in B\atop k\neq j}P^\lambda_\pi(X_{\tau_B}=k,X_{\tau_B\vee t(\lambda,\epsilon)}=j) + P^\lambda_\pi(X_{\tau_B}=j).
\end{split}
\end{equation}
By the strong Markov property, for any $k\in B$ and $k\neq j$,
\begin{equation}
\begin{split}
& P^\lambda_\pi(X_{\tau_B}=k,X_{\tau_B\vee t(\lambda,\epsilon)}=j) = P^\lambda_\pi(X_{\tau_B}=k)P^\lambda_k(X_{\tau_B\vee t(\lambda,\epsilon)-\tau_B}=j) \\
&\leq P^\lambda_k(X_{\tau_B\vee t(\lambda,\epsilon)-\tau_B}=j) \leq P^\lambda_k(E_k\leq\tau_B\vee t(\lambda,\epsilon)-\tau_B) \\
&\leq P^\lambda_k(E_k\leq t(\lambda,\epsilon)) = 1-e^{-q_k(\lambda)t(\lambda,\epsilon)},
\end{split}
\end{equation}
where $E_k$ is an exponential random variable with parameter $q_k(\lambda)$. Since $k\in B$, by Lemma \ref{prep}, we have $q_k(\lambda)t(\lambda,\epsilon)\rightarrow 0$ as $\lambda\rightarrow\infty$. Thus when $\lambda$ is sufficiently large, for any $k\in B$,
\begin{equation}
1-e^{-q_k(\lambda)t(\lambda,\epsilon)} \leq \epsilon.
\end{equation}
Thus we have
\begin{equation}
P^\lambda_\pi(X_{\tau_B}=k,X_{\tau_B\vee t(\lambda,\epsilon)}=j) \leq \epsilon.
\end{equation}
In view of \eqref{p1}, we obtain that
\begin{equation}\label{left}
P^\lambda_\pi(X_{t(\lambda,\epsilon)}=j) \leq P^\lambda_\pi(X_{\tau_B}=j) + |B|\epsilon.
\end{equation}
By Lemma \ref{prep}, when $\lambda$ is sufficiently large, we have $P^\lambda_\pi(\tau_B<t(\lambda,\epsilon))\geq 1-\epsilon$. Thus we have
\begin{equation}\label{p2}
\begin{split}
& P^\lambda_\pi(X_{\tau_B}=j) \leq P^\lambda_\pi(X_{\tau_B}=j,\tau_B<t(\lambda,\epsilon)) + \epsilon \\
&\leq P^\lambda_\pi(X_{t(\lambda,\epsilon)}=j) + P^\lambda_\pi(X_{\tau_B}=j,X_{t(\lambda,\epsilon)}\neq j,\tau_B<t(\lambda,\epsilon)) + \epsilon \\
&\leq P^\lambda_\pi(X_{t(\lambda,\epsilon)}=j) + P^\lambda_\pi(X_{\tau_B}=j,X_{\tau_B\vee t(\lambda,\epsilon)}\neq j) + \epsilon.
\end{split}
\end{equation}
Using the strong Markov property again, we obtain that
\begin{equation}
\begin{split}
& P^\lambda_\pi(X_{\tau_B}=j,X_{\tau_B\vee t(\lambda,\epsilon)}\neq j) = P^\lambda_\pi(X_{\tau_B}=j)P^\lambda_j(X_{\tau_B\vee t(\lambda,\epsilon)-\tau_B}\neq j) \\
&\leq P^\lambda_j(X_{\tau_B\vee t(\lambda,\epsilon)-\tau_B}\neq j) \leq P^\lambda_j(E_j\leq\tau_B\vee t(\lambda,\epsilon)-\tau_B) \\
&\leq P^\lambda_j(E_j\leq t(\lambda,\epsilon)) = 1-e^{-q_j(\lambda)t(\lambda,\epsilon)} \leq \epsilon.
\end{split}
\end{equation}
In view of \eqref{p2}, it follows that
\begin{equation}\label{right}
P^\lambda_\pi(X_{\tau_B}=j) \leq P^\lambda_\pi(X_{t(\lambda,\epsilon)}=j) + 2\epsilon.
\end{equation}
Combining \eqref{left} and \eqref{right}, we obtain that
\begin{equation}
|P^\lambda_\pi(X_{t(\lambda,\epsilon)}=j)-P^\lambda_\pi(X_{\tau_B}=j)| \leq (|B|+2)\epsilon.
\end{equation}
Thus we have
\begin{equation}
\sum_{j\in B}|P^\lambda_\pi(X_{t(\lambda,\epsilon)}=j)-P^\lambda_\pi(X_{\tau_B}=j)| \leq |B|(|B|+2)\epsilon.
\end{equation}
In addition, we have
\begin{equation}
\begin{split}
& \sum_{j\in A}|P^\lambda_\pi(X_{t(\lambda,\epsilon)}=j)-P^\lambda_\pi(X_{\tau_B}=j)|
= \sum_{j\in A}P^\lambda_\pi(X_{t(\lambda,\epsilon)}=j) \\
&= 1-\sum_{j\in B}P^\lambda_\pi(X_{t(\lambda,\epsilon)}=j)
\leq 1-\sum_{j\in B}(P^\lambda_\pi(X_{\tau_B}=j)-2\epsilon)
= 2|B|\epsilon.
\end{split}
\end{equation}
Thus when $\lambda$ is sufficiently large,
\begin{equation}
\sum_{j\in S}|P^\lambda_\pi(X_{t(\lambda,\epsilon)}=j)-P^\lambda_\pi(X_{\tau_B}=j)| \leq |B|(|B|+4)\epsilon.
\end{equation}
This implies the result of this theorem.
\end{proof}

\begin{definition}
$P^\lambda_\pi(X_{\tau_B}\in\cdot)$ is called the first-passage distribution of $B$ for the Markov chain $X$.
\end{definition}

The above theorem shows that when $\lambda$ is sufficiently large, given an arbitrarily small error $\epsilon > 0$, we can always find a small deterministic time $t(\lambda,\epsilon)>0$, such that the distribution of the Markov chain $X$ at time $t(\lambda,\epsilon)$ is close to the first-passage distribution of $B$ with an admissible error less than $\epsilon$. This clearly shows that when $\lambda$ is sufficiently large, the distribution of the Markov chain $X$ will be almost concentrated on the slow state space $B$ within a very short time. The first-passage distribution of $B$ can be calculated explicitly, as shown in the following lemma.
\begin{lemma}\label{passagedist}
Let $\pi=(\pi_A,\pi_B)$ be a probability distribution on the state space $S$. Then
\begin{equation}
P^\lambda_\pi(X_{\tau_B}\in\cdot)  = \pi_B + \pi_A(I-\Omega_{AA}(\lambda))^{-1}\Omega_{AB}(\lambda).
\end{equation}
\end{lemma}

\begin{proof}
For any $j\in B$,
\begin{equation}
\begin{split}
& P^\lambda_\pi(X_{\tau_B}=j)
= \pi_j + \sum_{i\in A}\pi_iP^\lambda_i(X_{\tau_B}=j) \\
&= \pi_j + \sum_{i\in A}\pi_i\sum_{n=0}^\infty\sum_{k\in A}P^\lambda_i(\xi_1,\cdots,\xi_n\in A,\xi_n=k,\xi_{n+1}=j) \\
&= \pi_j + \sum_{i\in A}\pi_i\sum_{n=0}^\infty\sum_{k\in A}[{\Omega_{AA}(\lambda)}^n]_{ik}\omega_{kj}(\lambda) \\
&= \pi_j + \sum_{i,k\in A}\pi_i[(I-\Omega_{AA}(\lambda))^{-1}]_{ik}\omega_{kj}(\lambda).
\end{split}
\end{equation}
This implies the result of this lemma.
\end{proof}

The above lemma shows that as $\lambda\rightarrow\infty$, the first-passage distribution of $B$ for the Markov chain $X$ will converge to the probability distribution
\begin{equation}\label{gammapi}
\gamma(\pi) = \pi_B + \pi_A(I-\Omega_{AA})^{-1}\Omega_{AB}.
\end{equation}
The next result is a direct corollary of Theorem \ref{shortterm1} and Lemma \ref{passagedist}.
\begin{corollary}\label{shortterm2}
Let $\pi=(\pi_A,\pi_B)$ be a probability distribution on the state space $S$ and let $\gamma(\pi) = \pi_B + \pi_A(I-\Omega_{AA})^{-1}\Omega_{AB}$. Then for any $\lambda>0$ and $\epsilon>0$, there exists $t(\lambda,\epsilon)>0$ such that $\lim_{\lambda\rightarrow\infty}t(\lambda,\epsilon) = 0$ and when $\lambda$ is sufficiently large,
\begin{equation}
d_{TV}(P^\lambda_\pi(X_{t(\lambda,\epsilon)}\in\cdot),\gamma(\pi)) \leq \epsilon.
\end{equation}
\end{corollary}

\begin{proof}
By Theorem \ref{shortterm1}, for any $\lambda>0$ and $\epsilon>0$, we can choose $t(\lambda,\epsilon)>0$ such that $\lim_{\lambda\rightarrow\infty}t(\lambda,\epsilon) = 0$ and when $\lambda$ is sufficiently large,
\begin{equation}
d_{TV}(P^\lambda_\pi(X_{t(\lambda,\epsilon)}\in\cdot),\pi_B + \pi_A(I-\Omega_{AA}(\lambda))^{-1}\Omega_{AB}(\lambda)) \leq \epsilon.
\end{equation}
Note that when $\lambda$ is sufficiently large,
\begin{equation}
d_{TV}(\pi_B + \pi_A(I-\Omega_{AA}(\lambda))^{-1}\Omega_{AB}(\lambda),\gamma(\pi)) \leq \epsilon.
\end{equation}
The rest of the proof follows from the triangle inequality of the total variation distance.
\end{proof}

When the initial distribution of the original chain $X$ is not concentrated on the slow state space $B$, we cannot expect that the distributions of the original chain $X$ and the reduced chain $Y$ are close to each other over the whole time axis when $\lambda$ is sufficiently large. However, we can prove that for any $h>0$, the distributions of the original chain $X$ and the reduced chain $Y$ are close to each other after time $h$ when $\lambda$ is sufficiently large.
\begin{theorem}\label{weaklocal}
Let $\pi$ be a probability distribution on the state space $S$ and let $\gamma(\pi) = \pi_B + \pi_A(I-\Omega_{AA})^{-1}\Omega_{AB}$. Then for any $0<h<T$,
\begin{equation}
\lim_{\lambda\rightarrow\infty}\sup_{h\leq t\leq T}d_{TV}(P^\lambda_\pi(X_t\in\cdot),P_{\gamma(\pi)}(Y_t\in\cdot)) = 0.
\end{equation}
\end{theorem}

\begin{proof}
We only need to prove that for any $j\in B$,
\begin{equation}
\lim_{\lambda\rightarrow\infty}\sup_{h\leq t\leq T}|P^\lambda_\pi(X_t=j)-P_{\gamma(\pi)}(Y_t=j)| = 0.
\end{equation}
By Corollary \ref{shortterm2}, for any $\lambda>0$ and $\epsilon>0$, we can choose $t(\lambda,\epsilon)>0$ such that $\lim_{\lambda\rightarrow\infty}t(\lambda,\epsilon) = 0$ and when $\lambda$ is sufficiently large,
\begin{equation}\label{general1}
d_{TV}(P^\lambda_\pi(X_{t(\lambda,\epsilon)}\in\cdot),\gamma(\pi)) \leq \epsilon.
\end{equation}
Thus when $\lambda$ is sufficiently large, we have $0<t(\lambda,\epsilon)<h$. Thus for any $h\leq t\leq T$,
\begin{eqnarray*}
&& |P^\lambda_\pi(X_t=j)-P_{\gamma(\pi)}(Y_t=j)| \\
&=& |\sum_{k\in S}P^\lambda_\pi(X_{t(\lambda,\epsilon)}=k)P^\lambda_k(X_{t-t(\lambda,\epsilon)}=j)-P_{\gamma(\pi)}(Y_t=j)| \\
&\leq& \sum_{k\in A}P^\lambda_\pi(X_{t(\lambda,\epsilon)}=k) + |\sum_{k\in B}P^\lambda_\pi(X_{t(\lambda,\epsilon)}=k)P^\lambda_k(X_{t-t(\lambda,\epsilon)}=j)-P_{\gamma(\pi)}(Y_t=j)| \\
&\leq& \sum_{k\in A}P^\lambda_\pi(X_{t(\lambda,\epsilon)}=k) + |\sum_{k\in B}(P^\lambda_\pi(X_{t(\lambda,\epsilon)}=k)-\gamma_k(\pi))P^\lambda_k(X_{t-t(\lambda,\epsilon)}=j)| \\
&& + |\sum_{k\in B}\gamma_k(\pi)(P^\lambda_k(X_{t-t(\lambda,\epsilon)}=j)-P_k(Y_t=j))| \\
&\leq& \sum_{k\in S}|P^\lambda_\pi(X_{t(\lambda,\epsilon)}=k)-\gamma_k(\pi)| + |\sum_{k\in B}\gamma_k(\pi)(P^\lambda_k(X_{t-t(\lambda,\epsilon)}=j)-P_k(Y_{t-t(\lambda,\epsilon)}=j))| \\
&& + |\sum_{k\in B}\gamma_k(\pi)(P_k(Y_{t-t(\lambda,\epsilon)}=j)-P_k(Y_t=j))|.
\end{eqnarray*}
In view of \eqref{general1}, we have
\begin{equation}\label{general2}
\sum_{k\in S}|P^\lambda_\pi(X_{t(\lambda,\epsilon)}=k)-\gamma_k(\pi)| = 2d_{TV}(P^\lambda_\pi(X_{t(\lambda,\epsilon)}\in\cdot),\gamma(\pi)) \leq 2\epsilon.
\end{equation}
By Theorem \ref{stronglocal}, when $\lambda$ is sufficiently large, for any $k\in B$,
\begin{equation}\label{general3}
|P^\lambda_k(X_{t-t(\lambda,\epsilon)}=j)-P_k(Y_{t-t(\lambda,\epsilon)}=j)| \leq \epsilon.
\end{equation}
Combining \eqref{general2} and \eqref{general3} and using Lemma \ref{timeclose}, we obtain that
\begin{equation}
\begin{split}
|P^\lambda_\pi(X_t=j)-P_{\gamma(\pi)}(Y_t=j)| &\leq 2\epsilon + \epsilon + \sum_{k\in B}\gamma_k(\pi)|P_k(Y_{t-t(\lambda,\epsilon)}=j)-P_k(Y_t=j)| \\
&\leq 3\epsilon + \sum_{k\in B}\gamma_k(\pi)(1-e^{-\gamma_kt(\lambda,\epsilon)}).
\end{split}
\end{equation}
Note that
\begin{equation}
\lim_{\lambda\rightarrow\infty}\sum_{k\in B}\gamma_k(\pi)(1-e^{-\gamma_kt(\lambda,\epsilon)}) = 0.
\end{equation}
Thus when $\lambda$ is sufficiently large,
\begin{equation}
\sum_{k\in B}\gamma_k(\pi)(1-e^{-\gamma_kt(\lambda,\epsilon)}) \leq \epsilon.
\end{equation}
Thus when $\lambda$ is sufficiently large, for any $h\leq t\leq T$,
\begin{equation}
|P^\lambda_\pi(X_t=j)-P_{\gamma(\pi)}(Y_t=j)| \leq 4\epsilon.
\end{equation}
This implies the result of this theorem.
\end{proof}

If we further assume that the reduced chain $Y$ is irreducible, then the conclusion of Theorem \ref{weaklocal} can be strengthened, as shown in the following theorem.
\begin{theorem}\label{weakglobal}
Assume that the reduced chain $Y$ is irreducible. Let $\pi$ be a probability distribution on the state space $S$ and let $\gamma(\pi) = \pi_B + \pi_A(I-\Omega_{AA})^{-1}\Omega_{AB}$. Then for any $h>0$,
\begin{equation}
\lim_{\lambda\rightarrow\infty}\sup_{t\geq h}d_{TV}(P^\lambda_\pi(X_t\in\cdot),P_{\gamma(\pi)}(Y_t\in\cdot)) = 0.
\end{equation}
\end{theorem}

\begin{proof}
The proof of this theorem is totally the same as that of Theorem \ref{weaklocal}.
\end{proof}

\begin{figure}[!htb]
\begin{center}
\centerline{\includegraphics[width=0.8\textwidth]{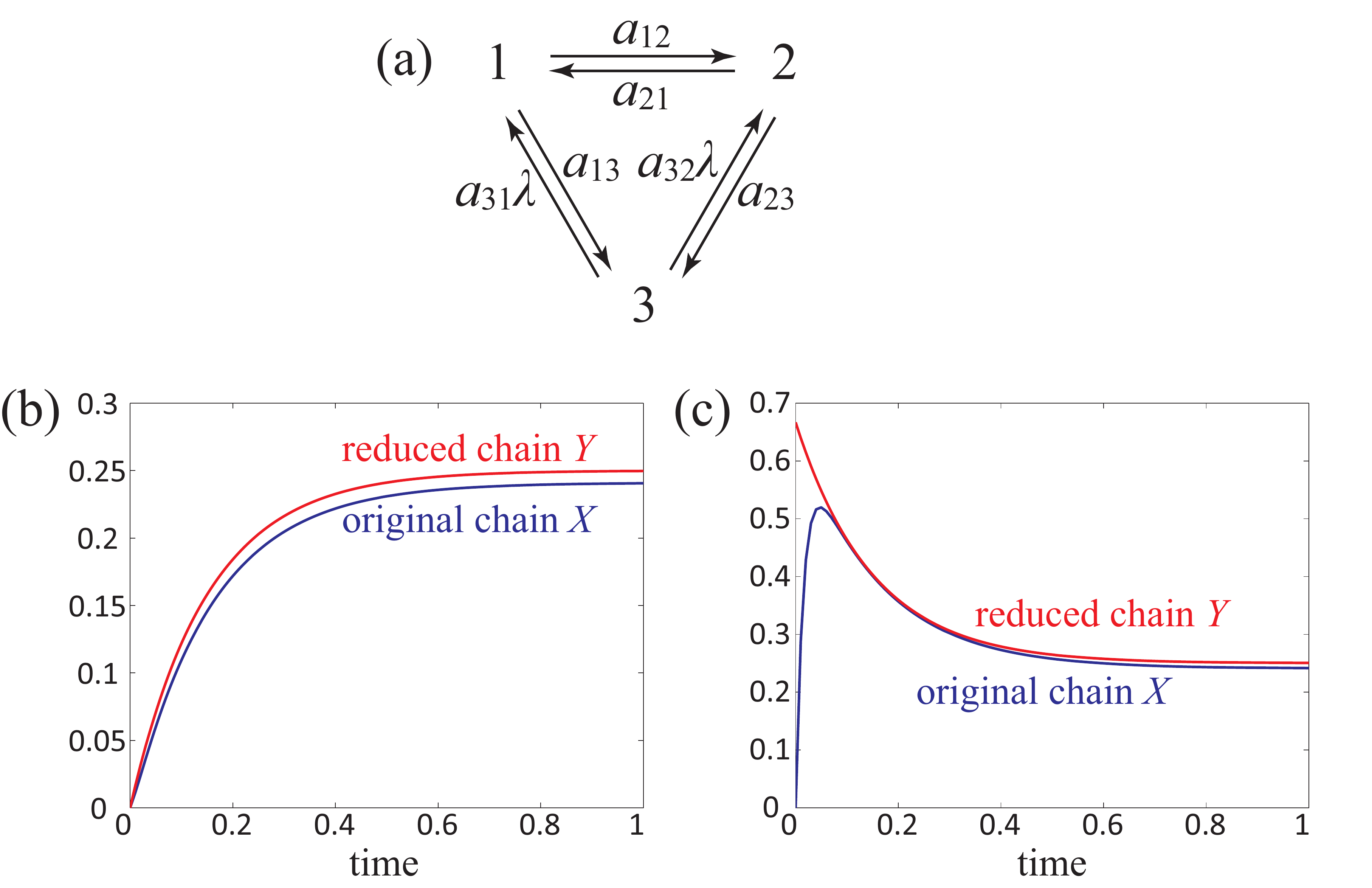}}
\caption{Illustration of the main results of this paper.}\label{threestatemodel}
\end{center}
\end{figure}

\begin{example}
Consider the three-state Markov chain $X$ illustrated in Figure \ref{threestatemodel}(a) for which two transition rates depend on $\lambda$ in a linear way and other transition rates are independent of $\lambda$. The transition rate matrix of the Markov chain $X$ is given by
\begin{equation}
Q(\lambda) =
\begin{pmatrix}
-(a_{12}+a_{13}) & a_{12} & a_{13} \\
a_{21} & -(a_{21}+a_{23}) & a_{23} \\
a_{31}\lambda & a_{32}\lambda & -(a_{31}+a_{32})\lambda
\end{pmatrix}.
\end{equation}
According to Definition \ref{states}, the fast state space is $A = \{3\}$ and the slow state space is $B = \{1,2\}$. It is easy to verify that the transition rate matrix $\Gamma$ of the reduced chain $Y$ is
\begin{equation}
\Gamma =
\begin{pmatrix}
-(a_{12}+\frac{a_{13}a_{32}}{a_{31}+a_{32}}) & a_{12}+\frac{a_{13}a_{32}}{a_{31}+a_{32}} \\
a_{21}+\frac{a_{23}a_{31}}{a_{31}+a_{32}} & -(a_{21}+\frac{a_{23}a_{31}}{a_{31}+a_{32}}) \\
 \end{pmatrix}.
\end{equation}

We shall now provide a visualized explanation of the main results of this paper. In the following discussion, we choose $a_{12} = 3$, $a_{13} = 6$, $a_{21} = 1$, $a_{23} = 1$, $a_{31} = 4$, $a_{32} = 2$, and $\lambda = 10$. With these parameters, the state transitions of the Markov chain $X$ have two separated time scales.

We first assume that the initial state of $X$ is state 2. In this case, the initial distribution of $X$ is concentrated on the slow state space $B$. By Theorem \ref{strongglobal}, the two transition probabilities, $P^\lambda_2(X_t=1)$ and $P_2(Y_t=1)$, should be close to each other over the whole time axis. This fact is illustrated in Figure \ref{threestatemodel}(b), where the blue and red lines represent the graphs of $P^\lambda_2(X_t=1)$ and $P_2(Y_t=1)$ as functions of time $t$, respectively.

We next assume that the initial state of $X$ is state 3. In this case, the initial distribution of $X$ is no longer concentrated on the slow state space $B$. By Theorem \ref{weakglobal}, the two transition probabilities, $P^\lambda_3(X_t=1)$ and $P_3(Y_t=1)$, should be close to each other after a very short time. This fact is illustrated in Figure \ref{threestatemodel}(c), where the blue and red lines represent the graphs of $P^\lambda_3(X_t=1)$ and $P_3(Y_t=1)$ as functions of time $t$, respectively.
\end{example}

\section{Relationship with the theory of singularly perturbed Markov chains}
In the previous literature, the asymptotic behavior of two-time-scales Markov chains is usually studied based on the model of singularly perturbed Markov chains. Yin, Zhang, and coworkers \cite{khashinskii1996asymptotic, yin2000asymptotic, yin2000singularly, yin2007singularly} have done a systematic study on the asymptotic behavior of singularly perturbed Markov chains using the approach of matched asymptotic expansions from singular perturbation theory, and these results have been organized into a textbook recently \cite{yin2012continuous}.

In this section, we shall discuss the relationship between our work and the theory of singularly perturbed Markov chains in detail. A continuous-time Markov chain $X$ is called a singularly perturbed Markov chain with weak and strong interactions if its transition rate matrix $Q(\lambda) = (q_{ij}(\lambda))$ depends on a parameter $\lambda>0$ in a linear way:
\begin{equation}
Q(\lambda) = \lambda\widetilde{Q}+\widehat{Q},
\end{equation}
where $\widetilde{Q}$ and $\widehat{Q}$ are two transition rate matrices. When $\lambda$ is very large, the transition rate matrix $\widetilde{Q}$ governs the rapidly changing components and the transition rate matrix $\widehat{Q}$ governs the slowly changing ones.

In this section, we assume that the singularly perturbed chain $X$ satisfies the assumptions of this paper. Let $A$ be the fast state space and let $B$ be the slow state space. For continence, we represent $\widetilde{Q}$ and $\widehat{Q}$ as block matrices
\begin{equation}
\widetilde{Q} =
\begin{pmatrix}
\widetilde{Q}_{AA} & \widetilde{Q}_{AB} \\ \widetilde{Q}_{BA} & \widetilde{Q}_{BB}
\end{pmatrix},\;\;\;
\widehat{Q} =
\begin{pmatrix}
\widehat{Q}_{AA} & \widehat{Q}_{AB} \\ \widehat{Q}_{BA} & \widehat{Q}_{BB}
\end{pmatrix}.
\end{equation}
In addition, we set $\widetilde{q}_i = -\widetilde{q}_{ii}$ and $\widehat{q}_i = -\widehat{q}_{ii}$. The next proposition follows immediately.
\begin{proposition}\label{slow}
Let $i$ be a state. Then $i$ is a slow state if and only if $i$ is an absorbing state of the transition rate matrix $\widetilde{Q}$.
\end{proposition}

\begin{proof}
By Definition \ref{states}, $i$ is a slow state if and only if $\lim_{\lambda\rightarrow\infty}q_i(\lambda) < \infty$. Note that
\begin{equation}
q_i(\lambda) = \lambda\widetilde{q}_i+\widehat{q}_i.
\end{equation}
This shows that $i$ is a slow state if and only if $\widetilde{q}_i = 0$, that is, $i$ is an absorbing state of the transition rate matrix $\widetilde{Q}$.
\end{proof}

To proceed, let
\begin{equation}
M_A(\lambda) = \diag(q_1(\lambda),\cdots,q_{|A|}(\lambda)) = \diag(\lambda\widetilde{q}_1+\widehat{q}_1,\cdots,\lambda\widetilde{q}_{|A|}+\widehat{q}_{|A|}).
\end{equation}
In view of \eqref{omega}, we have
\begin{equation}
\begin{split}
\Omega_{AB}(\lambda) &= M_A^{-1}(\lambda)(\lambda\widetilde{Q}_{AB}+\widehat{Q}_{AB}), \\
I-\Omega_{AA}(\lambda) &= -M_A^{-1}(\lambda)(\lambda\widetilde{Q}_{AA}+\widehat{Q}_{AA}).
\end{split}
\end{equation}
These two equations imply that
\begin{equation}\label{omega2}
\Omega_{AB} = \widetilde{M}_A^{-1}\widetilde{Q}_{AB},\;\;\;I-\Omega_{AA} = -\widetilde{M}_A^{-1}\widetilde{Q}_{AA},
\end{equation}
where $\widetilde{M}_A = \diag(\widetilde{q}_1,\cdots,\widetilde{q}_{|A|})$. The next proposition follows from the above equations.
\begin{proposition}\label{fast}
Let $i$ be a state. Then $i$ is a fast state if and only if $i$ is a transient state of the transition rate matrix $\widetilde{Q}$.
\end{proposition}

\begin{proof}
Let $i$ be a fast state. Assume that $i$ is a recurrent state of the transition rate matrix $\widetilde{Q}$. Let $C$ be the recurrent class of $\widetilde{Q}$ including $i$. It is easy to set that $C$ is a subset of the fast state space $A$. We arrange matters so that $C=\{1,\cdots,|C|\}$. Let $\widetilde{Q}_{CC}$ and $\Omega_{CC}$ be the matrices obtained from $\widetilde{Q}_{AA}$ and $\Omega_{AA}$ by retaining the rows and columns corresponding to the states in $C$, respectively. Since $C$ is a recurrent class of $\widetilde{Q}$, we have $\widetilde{Q}_{CC}1 = 0$. In view of \eqref{omega2}, we have
\begin{equation}
\Omega_{CC} = I+\widetilde{M}_C^{-1}\widetilde{Q}_{CC},
\end{equation}
where $\widetilde{M}_C = \diag(\widetilde{q}_1,\cdots,\widetilde{q}_{|C|})$. As a result, we have $\Omega_{CC}1 = 1$. This implies that $C$ is a recurrent class of the discrete-time Markov chain $\eta$. Thus we have $P_i(T_B<\infty) = 0$, which contradicts our assumption. This shows that $i$ must be a transient state of the transition rate matrix $\widetilde{Q}$.
\end{proof}

\begin{remark}
By Lemmas \ref{slow} and \ref{fast}, the states in the slow state space $B$ are absorbing states of the transition rate matrix $\widetilde{Q}$ and the states in the fast state space $A$ are transient states of the transition rate matrix $\widetilde{Q}$. Under the assumptions of this paper, the transition rate matrix $\widetilde{Q}$ cannot have a recurrent class with two or more states. Therefore, the framework of this paper generalizes the model of singularly perturbed Markov chains whose $\widetilde{Q}$ has only absorbing and transient states.
\end{remark}

In view of \eqref{omega2}, the transition rate matrix $\Gamma$ of the reduced chain $Y$ has the form of
\begin{equation}
\Gamma = Q_{BB}+Q_{BA}(I-\Omega_{AA})\Omega_{AB} = \widehat{Q}_{BB}+\widehat{Q}_{BA}\widetilde{Q}_{AA}^{-1}\widetilde{Q}_{AB}.
\end{equation}
Let $\pi = (\pi_A,\pi_B)$ be a probability distribution on the state space $S$. Owing to \eqref{gammapi}, as $\lambda\rightarrow\infty$, the first-passage distribution of $B$ for the Markov chain $X$ will converge to the probability distribution
\begin{equation}
\gamma(\pi) = \pi_B + \pi_A(I-\Omega_{AA})^{-1}\Omega_{AB} = \pi_B + \pi_A\widetilde{Q}_{AA}^{-1}\widetilde{Q}_{AB}.
\end{equation}
Applying the results of this paper to singularly perturbed Markov chains, we obtain the following two theorems.

\begin{theorem}
Let $\pi$ be a probability distribution concentrated on the slow state space $B$. Then for any $T>0$,
\begin{equation}
\lim_{\lambda\rightarrow\infty}\sup_{0\leq t\leq T}d_{TV}(P^\lambda_{\pi}(X_t\in\cdot),P_{\pi}(Y_t\in\cdot)) = 0.
\end{equation}
If the reduced chain $Y$ is irreducible, then
\begin{equation}
\lim_{\lambda\rightarrow\infty}\sup_{t\geq 0}d_{TV}(P^\lambda_{\pi}(X_t\in\cdot),P_{\pi}(Y_t\in\cdot)) = 0.
\end{equation}
\end{theorem}

\begin{theorem}
Let $\pi$ be a probability distribution on the state space $S$. Then for any $0<h<T$,
\begin{equation}
\lim_{\lambda\rightarrow\infty}\sup_{h\leq t\leq T}d_{TV}(P^\lambda_{\pi}(X_t\in\cdot),P_{\gamma(\pi)}(Y_t\in\cdot)) = 0.
\end{equation}
If the reduced chain $Y$ is irreducible, then for any $h>0$,
\begin{equation}
\lim_{\lambda\rightarrow\infty}\sup_{t\geq h}d_{TV}(P^\lambda_{\pi}(X_t\in\cdot),P_{\gamma(\pi)}(Y_t\in\cdot)) = 0.
\end{equation}
\end{theorem}

In the following discussion, we shall demonstrate that the results of this paper are consistent with the theory of singularly perturbed Markov chains. By the theory of singularly perturbed Markov chains \cite{yin2012continuous}, if the transition rate matrix $\widetilde{Q}$ has only absorbing states and transient states, then the distribution of the Markov chain $X$ at time $t$ will converge to the zero-order outer expansion $\phi(t) = (\phi_A(t),\phi_B(t))$ as $\lambda\rightarrow\infty$ for any $t>0$. In view of (4.86) and (4.88) in \cite{yin2012continuous}, we have $\phi_A(t) = 0$ and $\phi_B(t)$ is the solution to the following ordinary differential equation:
\begin{equation}\left\{
\begin{split}
\dot{\phi}_B(t) &= \phi_B(t)(\widehat{Q}_{BB}+\widehat{Q}_{BA}\widetilde{Q}_{AA}^{-1}\widetilde{Q}_{AB}) = \phi_B(t)\Gamma \\
\phi_B(0) &= \pi_B + \pi_A\widetilde{Q}_{AA}^{-1}\widetilde{Q}_{AB} = \gamma(\pi).
\end{split}\right.
\end{equation}
Therefore, it is easy to see that
\begin{equation}
\phi_B(t) = \gamma(\pi)e^{\Gamma t},
\end{equation}
which is exactly the distribution of the reduced chain $Y$ at time $t$ under the initial distribution $\gamma(\pi)$. Therefore, the theory of singularly perturbed Markov chains shows that for any $t>0$,
\begin{equation}
\lim_{\lambda\rightarrow\infty}d_{TV}(P_{\pi}^\lambda(X_t\in\cdot),P_{\gamma(\pi)}(Y_t\in\cdot)) = 0,
\end{equation}
which is consistent with the results of this paper. In this paper, we prove stronger results about the uniform convergence over finite time intervals and over the whole time axis.

At the end of this section, we make several remarks about the comparison between our work and the theory of singularly perturbed Markov chains.
\begin{remark}
In the theory of singularly perturbed Markov chains, the asymptotic behavior of the Markov chain $X$ is obtained using the approach of matched asymptotic expansions from singular perturbation theory. This approach is purely analytic and the probabilistic meaning of the zero-order outer expansion $\phi(t)$ is often not emphasized. In this paper, we use a purely probabilistic approach to study the asymptotic behavior of the Markov chain $X$. We prove that the distribution of the original chain $X$ will converge to that of the reduced chain $Y$ uniformly in time $t$ as $\lambda\rightarrow\infty$. The proof of Lemma \ref{complicated} explains why such convergence holds.
\end{remark}

\begin{remark}
In the theory of singularly perturbed Markov chains, the initial value $\phi(0)$ of the zero-order outer expansion is determined based on the so-called initial-value consistency condition (see (4.53) in \cite{yin2012continuous}). Therefore, the probabilistic meaning of the initial value $\phi(0)$ is not so clear. In this paper, we make it clear that the initial value $\phi(0)$ is exactly the initial distribution of the reduced chain $Y$, which the limit of the first-passage distribution of the slow state space $B$ for the Markov chain $X$ as $\lambda\rightarrow\infty$ (see \eqref{gammapi}).
\end{remark}

\begin{remark}
In the theory of singularly perturbed Markov chains, the zero-order outer expansion $\phi(t)$ acts as a good approximation for the distribution of the Markov chain $X$ when $t$ is bounded away from 0. If we are concerned with the asymptotic behavior of the Markov chain $X$ when $t$ is in a neighborhood of 0, an additional term called the initial-layer correction must be introduced. In this paper, we demonstrate that if the initial distribution is not concentrated on the slow state space $B$, then the converge of the original chain $X$ to the reduced chain $Y$ only holds when $t$ is bounded away from 0 (see Theorems \ref{stronglocal} and \ref{strongglobal}). This is in accordance with the theory of singularly perturbed Markov chains.
\end{remark}

\section{Detailed proofs}
In this section, we shall give the proof of Lemma \ref{complicated}.
\begin{proof}[Proof of Lemma \ref{complicated}]
By the definition of the transition rate matrix $\Gamma$, we have
\begin{equation}
\gamma_{kj} = q_{kj}+\sum_{n=0}^\infty\sum_{l,m\in A}q_{kl}f_{lm}^{(n)}\omega_{mj},\;\;\;
\gamma_j = q_j-\sum_{n=0}^\infty\sum_{l,m\in A}q_{jl}f_{lm}^{(n)}\omega_{mj}.
\end{equation}
Thus for any $\epsilon>0$, we can choose a sufficiently large $N$, such that for any $k,j\in B$,
\begin{equation}
\begin{split}
\left|\gamma_{kj} - \left(q_{kj}+\sum_{n=0}^{N-2}\sum_{l,m\in A}q_{kl}f_{lm}^{(n)}\omega_{mj}\right)\right| &< \frac{\epsilon}{2}, \\
\left|\gamma_j - \left(q_j-\sum_{n=0}^{N-2}\sum_{l,m\in A}q_{jl}f_{lm}^{(n)}\omega_{mj}\right)\right| &< \frac{\epsilon}{2}.
\end{split}
\end{equation}
Thus there exists $\lambda_1>0$, such that for any $\lambda>\lambda_1$,
\begin{equation}\label{rateapp}
\begin{split}
\left|\gamma_{kj} - \left(q_{kj}(\lambda)+\sum_{n=0}^{N-2}\sum_{l,m\in A}q_{kl}(\lambda)f_{lm}^{(n)}(\lambda)\omega_{mj}(\lambda)\right)\right| &< \epsilon, \\
\left|\gamma_j - \left(q_j(\lambda)-\sum_{n=0}^{N-2}\sum_{l,m\in A}q_{jl}(\lambda)f_{lm}^{(n)}(\lambda)\omega_{mj}(\lambda)\right)\right| &< \epsilon.
\end{split}
\end{equation}
Recall that the constant $M$ is defined as
\begin{equation}
M = \max_{j\in B}{q_j}+\max_{j\in B}{\gamma_j}+1.
\end{equation}
Thus there exists $\lambda_2>0$, such that for any $\lambda>\lambda_2$ and $j\in B$, we have $1<M$, $q_j(\lambda)<M$, and $\gamma_j<M$. Thus for any $h<1/M$, $\lambda>\lambda_2$, and $j\in B$, we have
\begin{equation}
h<1,\;\;\;q_j(\lambda)h < 1,\;\;\;\gamma_jh < 1.
\end{equation}
Let $S_1,S_2,\cdots$ be the holding times of the Markov chain $X$ and let $J_n = \sum_{i=1}^nS_i$. Then $J_1,J_2,\cdots$ are the jump times of the Markov chain $X$. By the semigroup property, for any $i,j\in B$,
\begin{equation}
\begin{split}
& P^\lambda_i(X_{t+h}=j) \geq \sum_{k\in B}P^\lambda_i(X_t=k)P^\lambda_k(X_h=j) \\
&= \sum_{n=0}^\infty\sum_{k\in B}P^\lambda_i(X_t=k)P^\lambda_k(X_h=j,J_n\leq h<J_{n+1}) = A_1+A_2+A_3,
\end{split}
\end{equation}
where
\begin{equation}
\begin{split}
A_1 &= \sum_{k\in B}P^\lambda_i(X_t=k)P^\lambda_k(X_h=j,J_1>h), \\
A_2 &= \sum_{k\in B}P^\lambda_i(X_t=k)P^\lambda_k(X_h=j,J_1\leq h<J_2), \\
A_3 &= \sum_{n=2}^\infty\sum_{k\in B}P^\lambda_i(X_t=k)P^\lambda_k(X_h=j,J_n\leq h<J_{n+1}).
\end{split}
\end{equation}
In the following proof, we shall estimate $A_1$, $A_2$, and $A_3$, respectively. By Lemma \ref{inequality}, we have
\begin{equation}\label{part1}
A_1 = P^\lambda_i(X_t=j)P^\lambda_j(J_1>h) = P^\lambda_i(X_t=j)e^{-q_j(\lambda)h} \geq P^\lambda_i(X_t=j)(1-q_j(\lambda)h).
\end{equation}
Moreover, we have
\begin{equation}
\begin{split}
A_2 &= \sum_{k\in B}P^\lambda_i(X_t=k)P^\lambda_k(\xi_1=j,J_1\leq h<J_2) \\
&= \sum_{k\in B}P^\lambda_i(X_t=k)P^\lambda_k(E_k\leq h<E_k+E_j)\omega_{kj}(\lambda) \\
&\geq \sum_{k\in B}P^\lambda_i(X_t=k)P^\lambda_k(E_k\leq h)P^\lambda_k(E_j>h)\omega_{kj}(\lambda) \\
&= \sum_{k\in B}P^\lambda_i(X_t=k)(1-e^{-q_k(\lambda)h})e^{-q_j(\lambda)h}\omega_{kj}(\lambda).
\end{split}
\end{equation}
where $E_k$ and $E_j$ are two independent exponential random variables with parameters $q_k(\lambda)$ and $q_j(\lambda)$, respectively. By Lemma \ref{inequality}, for any $\lambda>\lambda_2$ and $j\in B$,
\begin{equation}
e^{-q_j(\lambda)h} - (1-q_j(\lambda)h) \leq q_j(\lambda)^2h^2 \leq M^2h^2.
\end{equation}
This implies that
\begin{equation}\label{impine}
1-e^{-q_j(\lambda)h} \geq q_j(\lambda)h - M^2h^2.
\end{equation}
Thus we obtain that
\begin{equation}\label{part2}
\begin{split}
A_2 &\geq \sum_{k\in B}P^\lambda_i(X_t=k)(q_k(\lambda)h-M^2h^2)(1-q_j(\lambda)h)\omega_{kj}(\lambda) \\
&\geq \sum_{k\in B}P^\lambda_i(X_t=k)(q_k(\lambda)h-M^2h^2-q_k(\lambda)q_j(\lambda)h^2)\omega_{kj}(\lambda) \\
&\geq \sum_{k\in B}P^\lambda_i(X_t=k)(q_k(\lambda)h-2M^2h^2)\omega_{kj}(\lambda).
\end{split}
\end{equation}
Note that for any $n\geq 2$ and $k\in B$,
\begin{equation}\label{part3pre}
\begin{split}
& P^\lambda_k(X_h=j,J_n\leq h<J_{n+1}) = P^\lambda_k(\xi_n=j,J_n\leq h<J_{n+1}) \\
&\geq \sum_{k_1,\cdots,k_{n-1}\in A}P^\lambda_k(\xi_1=k_1,\cdots,\xi_{n-1}=k_{n-1},\xi_n=j,J_n\leq h<J_{n+1}) \\
&= \sum_{k_1,\cdots,k_{n-1}\in A}P^\lambda_k(E_{k,k_1,\cdots,k_{n-1}}\leq h<E_{k,k_1,\cdots,k_{n-1}}+E_j)\omega_{kk_1}(\lambda)\cdots\omega_{k_{n-1}j}(\lambda) \\
&\geq \sum_{k_1,\cdots,k_{n-1}\in A}P^\lambda_k(E_{k,k_1,\cdots,k_{n-1}}\leq h)P^\lambda_k(E_j>h)\omega_{kk_1}(\lambda)\cdots\omega_{k_{n-1}j}(\lambda) \\
&= \sum_{k_1,\cdots,k_{n-1}\in A}P^\lambda_k(E_{k,k_1,\cdots,k_{n-1}}\leq h)e^{-q_j(\lambda)h}\omega_{kk_1}(\lambda)\cdots\omega_{k_{n-1}j}(\lambda),
\end{split}
\end{equation}
where $E_{k,k_1,\cdots,k_{n-1}}$ is the sum of independent exponential random variables with parameters $q_k(\lambda)$, $q_{k_1}(\lambda),\cdots,q_{k_{n-1}}(\lambda)$, respectively, and $E_j$ is an exponential random variable with parameter $q_j(\lambda)$ which is independent of $E_{k,k_1,\cdots,k_{n-1}}$. Since $k\in B$ and $k_1,\cdots,k_{n-1}\in A$, by Slutsky's theorem, both $E_k$ and $E_{k,k_1,\cdots,k_{n-1}}$ will converge in distribution to an exponential random variable with parameter $q_k$ as $\lambda\rightarrow\infty$. Thus there exists $\lambda_3>0$, such that for any $\lambda>\lambda_3$, $n\leq N$, $k\in B$, and $k_1,\cdots,k_{n-1}\in A$,
\begin{equation}
\left|P^\lambda_k(E_{k,k_1,\cdots,k_{n-1}}\leq h)-P^\lambda_k(E_k\leq h)\right| \leq \epsilon.
\end{equation}
In view of \eqref{impine} and \eqref{part3pre}, it follows that
\begin{equation}\label{part3}
\begin{split}
& A_3 \geq \sum_{n=2}^N\sum_{k\in B}P^\lambda_i(X_t=k)P^\lambda_k(X_h=j,J_n\leq h<J_{n+1}) \\
&\geq \sum_{n=2}^N\sum_{k_1,\cdots,k_{n-1}\in A}\sum_{k\in B}P^\lambda_i(X_t=k)(P^\lambda_k(E_k\leq h)-\epsilon)(1-q_j(\lambda)h)\omega_{kk_1}(\lambda)\cdots\omega_{k_{n-1}j}(\lambda) \\
&= \sum_{n=2}^N\sum_{l,m\in A}\sum_{k\in B}P^\lambda_i(X_t=k)(1-e^{-q_k(\lambda)h}-\epsilon)(1-q_j(\lambda)h)\omega_{kl}(\lambda)f_{lm}^{(n-2)}(\lambda)\omega_{mj}(\lambda) \\
&\geq \sum_{n=2}^N\sum_{l,m\in A}\sum_{k\in B}P^\lambda_i(X_t=k)(q_k(\lambda)h-M^2h^2-\epsilon)(1-q_j(\lambda)h)\omega_{kl}(\lambda)f_{lm}^{(n-2)}(\lambda)\omega_{mj}(\lambda) \\
&\geq \sum_{n=2}^N\sum_{l,m\in A}\sum_{k\in B}P^\lambda_i(X_t=k)(q_k(\lambda)h-2M^2h^2-\epsilon)\omega_{kl}(\lambda)f_{lm}^{(n-2)}(\lambda)\omega_{mj}(\lambda).
\end{split}
\end{equation}
Combining \eqref{part1}, \eqref{part2}, and \eqref{part3}, we obtain that
\begin{eqnarray*}
&& P^\lambda_i(X_{t+h}=j) \\
&\geq& P^\lambda_i(X_t=j)(1-q_j(\lambda)h) + \sum_{k\in B}P^\lambda_i(X_t=k)(q_k(\lambda)h-2M^2h^2-\epsilon)\omega_{kj}(\lambda) \\
&& +\sum_{k\in B}P^\lambda_i(X_t=k)\sum_{n=2}^N\sum_{l,m\in A}(q_k(\lambda)h-2M^2h^2-\epsilon)\omega_{kl}(\lambda)f_{lm}^{(n-2)}(\lambda)\omega_{mj}(\lambda) \\
&=& P^\lambda_i(X_t=j)\left(1-q_j(\lambda)h+\sum_{n=2}^N\sum_{l,m\in A}q_{jl}(\lambda)f_{lm}^{(n-2)}(\lambda)\omega_{mj}(\lambda)h\right) \\
&& +\sum_{k\in B\atop k\neq j}P^\lambda_i(X_t=k)\left(q_{kj}(\lambda)h+\sum_{n=2}^N\sum_{l,m\in A}q_{kl}(\lambda)f_{lm}^{(n-2)}(\lambda)\omega_{mj}(\lambda)h\right) \\
&& -(2M^2h^2+\epsilon)\sum_{k\in B}P^\lambda_i(X_t=k)\left(\omega_{kj}(\lambda)+\sum_{n=2}^N\sum_{l,m\in A}\omega_{kl}(\lambda)f_{lm}^{(n-2)}(\lambda)\omega_{mj}(\lambda)\right).
\end{eqnarray*}
By the recurrence of the Markov chain $X$, we have
\begin{equation}
\begin{split}
& \sum_{k\in B}P^\lambda_i(X_t=k)\left(\omega_{kj}(\lambda) + \sum_{n=2}^{N}\sum_{l,m\in A}\omega_{kl}(\lambda)f_{lm}^{(n-2)}(\lambda)\omega_{mj}(\lambda)\right) \\
&\leq \sum_{k\in B}P^\lambda_i(X_t=k)P^\lambda_k\left(\bigcup_{n=1}^\infty\{\xi_n=j\}\right) \leq \sum_{k\in B}P^\lambda_i(X_t=k) \leq 1.
\end{split}
\end{equation}
Thus we obtain that
\begin{eqnarray}\label{connection}
&& P^\lambda_i(X_{t+h}=j) \nonumber \\
&\geq& P^\lambda_i(X_t=j)\left\{1-\left(q_j(\lambda)-\sum_{n=0}^{N-2}\sum_{l,m\in A}q_{jl}(\lambda)f_{lm}^{(n)}(\lambda)\omega_{mj}(\lambda)\right)h\right\} \\
&& +\sum_{k\in B\atop k\neq j}P^\lambda_i(X_t=k)\left(q_{kj}(\lambda)+\sum_{n=0}^{N-2}\sum_{l,m\in A}q_{kl}(\lambda) f_{lm}^{(n)}(\lambda)\omega_{mj}(\lambda)\right)h - (2M^2h^2+\epsilon). \nonumber
\end{eqnarray}
In view of \eqref{rateapp} and \eqref{connection}, for any $\lambda>\max\{\lambda_1,\lambda_2,\lambda_3\}$,
\begin{eqnarray}\label{finalX}
&& P^\lambda_i(X_{t+h}=j) \nonumber\\
&\geq& P^\lambda_i(X_t=j)(1-\gamma_jh-\epsilon h) + \sum_{k\in B\atop k\neq j}P^\lambda_i(X_t=k)(\gamma_{kj}h-\epsilon h) - (2M^2h^2+\epsilon) \\
&\geq& P^\lambda_i(X_t=j)(1-\gamma_jh) + \sum_{k\in B\atop k\neq j}P^\lambda_i(X_t=k)\gamma_{kj}h - 2(M^2h^2+\epsilon) \nonumber.
\end{eqnarray}
On the other hand, by the semigroup property, we have
\begin{eqnarray*}
&& P_i(Y_{t+h}=j) = \sum_{k\in B}P_i(Y_t=k)P_k(Y_h=j) \\
&=& \sum_{k\in B}P_i(Y_t=k)P_k(Y_h=j,J_1>h) + \sum_{k\in B}P_i(Y_t=k)P_k(Y_h=j,J_1\leq h<J_2) \\
&& +\sum_{k\in B}P_i(Y_t=k)P_k(Y_h=j,J_2\leq h).
\end{eqnarray*}
By Lemma \ref{inequality}, for any $\lambda>\lambda_2$ and $j\in B$,
\begin{equation}
e^{-\gamma_jh} - (1-\gamma_jh) \leq \gamma_j^2h^2 \leq M^2h^2.
\end{equation}
This implies that
\begin{equation}
e^{-\gamma_jh} \leq 1-\gamma_jh+M^2h^2.
\end{equation}
Thus we obtain that
\begin{equation}\label{part1Y}
\begin{split}
& \sum_{k\in B}P_i(Y_t=k)P_k(Y_h=j,J_1>h) = P_i(Y_t=j)P_j(J_1>h) \\
&= P_i(Y_t=j)e^{-\gamma_jh} \leq P_i(Y_t=j)(1-\gamma_jh+M^2h^2).
\end{split}
\end{equation}
Let $\sigma = \{\sigma_n:n\geq 0\}$ be the jump chain of the induced chain $Y$ and let $P = (p_{ij})$ be the jump matrix of the reduced chain $Y$. By Lemma \ref{inequality}, we have
\begin{equation}\label{part2Y}
\begin{split}
& \sum_{k\in B}P_i(Y_t=k)P_k(Y_h=j,J_1\leq h<J_2) = \sum_{k\in B}P_i(Y_t=k)P_k(\sigma_1=j,J_1\leq h<J_2) \\
&\leq \sum_{k\in B}P_i(Y_t=k)P_k(F_k\leq h)p_{kj} = \sum_{k\in B}P_i(Y_t=k)(1-e^{-\gamma_kh})p_{kj} \\
&\leq \sum_{k\in B}P_i(Y_t=k)(\gamma_{k}h)p_{kj} = \sum_{k\in B\atop k\neq j}P_i(Y_t=k)\gamma_{kj}h,
\end{split}
\end{equation}
where $F_k$ is an exponential random variable with parameter $\gamma_k$. By Lemma \ref{inequality}, we obtain that
\begin{equation}\label{part3Y}
\begin{split}
& \sum_{k\in B}P_i(Y_t=k)P_k(Y_h=j,J_2\leq h) \leq \sum_{k,l\in B}P_i(Y_t=k)P_k(\sigma_1=l,J_2\leq h) \\
&= \sum_{k,l\in B}P_i(Y_t=k)P_k(F_k+F_l\leq h)p_{kl} \leq \sum_{k,l\in B}P_i(Y_t=k)P_k(F_k\leq h)P_k(F_l\leq h)p_{kl} \\
&\leq \sum_{k,l\in B}P_i(Y_t=k)(\gamma_kh)(\gamma_lh)p_{kl} \leq \sum_{k,l\in B}P_i(Y_t=k)M^2h^2p_{kl} = M^2h^2,
\end{split}
\end{equation}
where $F_k$ and $F_l$ are two independent exponential random variables with parameters $\gamma_k$ and $\gamma_l$, respectively. Combining \eqref{part1Y}, \eqref{part2Y}, and \eqref{part3Y}, we obtain that
\begin{equation}\label{finalY}
\begin{split}
P_i(Y_{t+h}=j) &\leq P_i(Y_t=j)(1-\gamma_jh+M^2h^2) + \sum_{k\in B\atop k\neq j}P_i(Y_t=k)\gamma_{kj}h + M^2h^2 \\
&\leq P_i(Y_t=j)(1-\gamma_jh) + \sum_{k\in B\atop k\neq j}P_i(Y_t=k)\gamma_{kj}h + 2M^2h^2.
\end{split}
\end{equation}
In view of the assumption of this lemma, for any $i,j\in B$,
\begin{equation}
P^\lambda_i(X_t=j)-P_i(Y_t=j) \geq -\eta.
\end{equation}
Combining \eqref{finalX} and \eqref{finalY}, we obtain that
\begin{eqnarray*}
&& P^\lambda_i(X_{t+h}=j) \\
&\geq& (P_i(Y_t=j)-\eta)(1-\gamma_jh) + \sum_{k\in B\atop k\neq j}(P_i(Y_t=k)-\eta)\gamma_{kj}h - 2(M^2h^2+\epsilon) \\
&\geq& P_i(Y_{t+h}=j) - (1-\gamma_jh+\sum_{k\in B\atop k\neq j}\gamma_{kj}h)\eta - (4M^2h^2+2\epsilon) \\
&\geq& P_i(Y_{t+h}=j) -(1+|B|Mh)\eta - (4M^2h^2+2\epsilon).
\end{eqnarray*}
This completes the proof of this lemma.
\end{proof}

\section*{Acknowledgements}
The author gratefully acknowledges Professor Da-Quan Jiang for supporting my research on the present work and gratefully acknowledges the anonymous reviewers for their valuable comments and suggestions.

\setlength{\bibsep}{5pt}
\small\bibliographystyle{unsrt}
\bibliography{reduction}
\end{document}